\documentclass[pre,preprint,showpacs,floatfix]{revtex4-1}

\usepackage{setspace}
\usepackage{multirow}
\usepackage{latexsym}
\usepackage{amssymb,amsmath,amsfonts}
\usepackage{bm,bbm}
\usepackage{graphicx,epsfig}
\usepackage{caption, subcaption}
\usepackage{url}

\usepackage{fullpage}
\usepackage{cases}

\usepackage{booktabs}
\usepackage{paralist}
\usepackage{mathtools}
\usepackage{multirow}
\usepackage{color}
\usepackage{gensymb}
\usepackage{scalerel}

\usepackage{amsthm} % For theorems
\newtheorem{theorem}{Theorem}

\newcommand{\bfs}{\mathbf{s}} % For sample points
\newcommand{\bfz}{\mathbf{z}} % For prediction points

\newcommand{\bfr}{\mathbf{r}}
\newcommand{\Do}{\mathcal{D}}

\newcommand{\E}{\mathrm{e}} % exponential
\newcommand{\Iden}{\mathbf{I}_{N}}

\newcommand{\R}{\mathbb{R}}
\newcommand{\Rd}{\mathbb{R}^{d}}
\newcommand{\Zd}{\mathbb{Z}^{d}}
\newcommand{\bfh}{\mathbf{h}}
\newcommand{\rmx}{{\rm x}}
\newcommand{\Or}{\mathcal{O}}

\newcommand{\bfso}{{\bfz}_p}
\newcommand{\hx}{\hat{x}}

\newcommand{\Pd}{\mathcal{P}}  % Spatial domain of prediction (grid)
\newcommand{\HH}{{\mathcal H}}  % SLI energy function
\newcommand{\hHH}{{\mathcal {\hat{H}}}}

\newcommand{\bmthe}{{\bm \theta}}  % parameter vector
\newcommand{\xo}{x}
\newcommand{\bXo}{\mathbf{x}}

\newcommand{\xsam}{\bXo_{\mathrm S}}

\newcommand{\Samp}{{\mathbb{S}_N}}

\newcommand{\mx}{m_{\rmx}}

\newcommand{\bfmx}{\mathbf{\mx}}

\newcommand{\la}{\lambda}

\newcommand{\beq}{\begin{equation}}
\newcommand{\eeq}{\end{equation}}

\begin{document}

\title{Stochastic Local Interaction Model: Geostatistics without Kriging}

\author{Dionissios T. Hristopulos}
 \email{dionisi@mred.tuc.gr}   %optional
\author{Andreas Pavlides}
\author{Vasiliki D. Agou}
\author{Panagiota Gkafa}
\affiliation{School of Mineral Resources Engineering, Technical University of Crete,
Chania 73100, Greece}

\date{\today}

\begin{abstract}
Classical geostatistical methods face serious computational challenges if they are confronted with large sets of spatially distributed data. We present a simplified stochastic local interaction (SLI) model for  computationally efficient spatial prediction that can handle large data. The SLI method constructs a spatial interaction matrix (precision matrix) that accounts for the data values, their locations, and the sampling density variations without user input. We show that this precision matrix is strictly positive definite. The SLI approach does not require matrix inversion for parameter estimation, spatial prediction, and uncertainty estimation, leading to computational procedures that are significantly less intensive computationally than kriging. The precision matrix involves compact kernel functions (spherical, quadratic, etc.) which enable the application of sparse matrix methods,  thus improving computational time and memory requirements. We investigate the proposed SLI method with a data set that includes approximately 11500 drill-hole data of coal thickness from Campbell County (Wyoming, USA).  We also compare SLI with ordinary kriging (OK) in terms of estimation performance, using cross validation analysis, and computational time. According to the validation measures used, SLI performs slightly better in estimating seam thickness than OK in terms of cross-validation measures. In terms of computation time, SLI prediction is 3 to 25 times (depending on the size of the kriging neighborhood) faster than OK for the same grid size.
\end{abstract}



\maketitle

\newpage

%%%%%%%%%%%%%%%%%%%%%%%%%%%%%%%%%%%%%%%%%%%%%%%%%%%%%%%%%%%%%%%%%%%%%%%%%%
\section{Introduction} \label{sec:Intro}

Kriging and its various flavors are the golden rule in geostatistics~\cite{Christakos92,Cressie93,Wackernagel03,Olea12, Chiles12} and in mining applications of geostatistics in particular~\cite{Kyriakidis04,Armstrong98,Dimitrakopoulos07}. However, the application of kriging to large datasets faces  difficulties due to the computationally demanding inversion of the covariance matrix~\cite{Sun12,Zhong16}. Hence, various shortcuts and approximations have been developed that allow working with large datasets including the use of a finite search radius~\cite{Rivoirard87}, tapered covariance functions~\cite{Furrer06}, and fixed rank kriging~\cite{Cressie08}.
For reviews of the problem with large data see~\cite{Sun12,Lasinio13}.
For data that are collected on rectangular grids there is an efficient solution which is based on Gaussian Markov random fields (GMRFs)~\cite{Rue05}. However, this approach cannot be directly extended to scattered data, although it is possible to generalize GMRFs by linking them to stochastic partial differential equations~\cite{Rue11}.

The motivation for  \emph{stochastic local interaction} (SLI) models is the possibility, known in the statistical field theories of physics and in Markov random field models, to construct spatial and spatiotemporal dependence based on energy functions with \emph{local interactions}. In statistical physics this idea is present in the widely-used \emph{Boltzmann-Gibbs distributions} that include the Gaussian field theory and the binary Ising model among others, e.g.~\cite{Mussardo10}.

Local interactions allow building sparse and explicit precision (inverse covariance) matrices for the spatial dependence. Hence, computationally fast  algorithms can be designed for \emph{spatial prediction} based on the conditional mean. The uncertainty of the estimates is easily assessed by means of the  \emph{conditional variance} which follows directly from the precision matrix. The explicit form of the precision matrix implies independence from the curse of covariance matrix inversion, which plagues the application of geostatistical methods to large datasets. Hence, such models can be used as computationally efficient alternatives to kriging.

An SLI model that involves gradient and curvature terms was presented in~\cite{dth15cg}.
This model has a precision matrix that is not demonstrably non-negative definite, even though it is in practice well-defined in most cases. Herein we present a simpler version of the model which does not include the curvature term and possesses a positive definite precision matrix.

The SLI model can be viewed as an extension of \emph{Gaussian Markov random fields} (GMRFs)~\cite{Rue05} to scattered data by means of kernel functions.
From a different viewpoint, the SLI model can be viewed as a discrete analogue of Spartan spatial random fields which are defined in continuum space~\cite{Hristopulos2003,Hristopulos2007,dth15}.
In order to accommodate random sampling geometries, the SLI model incorporates ideas from machine learning such as the use of kernel functions. The range of the kernels is determined by a set of local bandwidths which is learned from the sample data.
Hence, SLI prediction uses only a subset of the data for prediction at any given point. In kriging this restriction is imposed by empirically determining a search neighborhood by trial and error. In the case of SLI, the range of the kernel is automatically adjusted by the algorithm (see Section~\ref{ssec:kernel_weights}).  An additional advantage of SLI over kriging is that the former does not require covariance matrix inversion to compute the prediction and prediction uncertainty.

The remainder of this manuscript is structured as follows: Section~\ref{sec:Methods} presents the SLI methodology, including the SLI model,  the equations for prediction and conditional standard deviation, and the leave-one-out cross validation procedure for model parameter estimation. In Section~\ref{Sec:Data} the Campbell coal dataset is briefly described. Section~\ref{sec:SLI-estim-gilette} presents the SLI model parameter estimation, while Section~\ref{sec:OrdKrig} focuses on the estimation of the variogram for ordinary kriging. Section~\ref{sec:estimation-coal} provides the SLI-based and kriging-based estimates of coal resources and the associated uncertainties. Section~\ref{sec:cross-validation} studies the stability of the SLI estimates and  compares the SLI and ordinary kriging predictive performance using leave-P-out cross validation.  Finally, in Section~\ref{sec:Conclusions} we present our conclusions, discuss potential applications of the method and analogies with other approaches, and outline directions for future research.

%%%%%%%%%%%%%%%%%%%%%%%%%%%%%%%%%%%%%

\section{Methodology} \label{sec:Methods}

\subsection{Notation}

In the following,    $\xsam \equiv (\xo_{1}, \ldots, \xo_{N})^{\top}$ denotes the sample values at the locations of the sampling set $\Samp=\{ \bfs_{1}, \ldots, \bfs_{N} \}$, where $\bfs_{i} \in \Do \subset \Rd, \, i=1, \ldots, N$. The superscript $\top$ denotes the matrix transpose.  The prediction set involves the points $\bfso, p=1, \ldots, P$ which typically correspond to the nodes of a \emph{map grid} $\Pd \subset \Zd$.  In general, $\Samp$ and $\Pd$ are disjoint sets.
The SLI optimal prediction at these points will be denoted by $\hx_{p}$.
Assuming a constant mean, $\mx$,  fluctuations around the mean will be denoted by primes, i.e.,
$\xo'_{i} = \xo_{i} - \mx$. In case the mean is not homogeneous, a trend function is subtracted to obtain the fluctuations. The SLI model is defined is terms of the fluctuations. However, the constant mean or the trend function do not need to be known \emph{a priori}, since they can be determined in the estimation stage.

\subsection{Kernel Functions} \label{ssec:KernelFun}

The SLI formulation uses ideas  from \emph{kernel regression}~\cite{Nadaraya64,Watson64} in order to express the local spatial dependence. The interactions between neighboring points are expressed in terms of suitably selected weighting functions, which  are supplied by \emph{kernel} functions.

A kernel (weighting) function $K(\bfs_n, \bfs_m): \Rd \times \Rd \to \R$ is a non-negative function that assigns a real value
to any pair of points $\bfs_{n}$ and $\bfs_{m}$.
The range of the kernel is determined by a characteristic length scale known as the \emph{bandwidth} and denoted by $h$.
The bandwidth should tend to zero as $N \to \infty$, so that for large and dense datasets emphasis is placed on points that are close to the target site.
 For $\bfs_n, \bfs_m \in \Rd$ let  $u = \| \bfs_n - \bfs_m \|/h$
 represent the normalized  distance.
 A kernel function is (a) non-negative, i.e.,  $K(u) \ge 0$  for all  $u$; (b) symmetric, i.e., $K(-u) = K(u)$ for all $u \in \R_{0,+}$; (c)
 maximized at $u =0$; (d)  a continuous function of $u$.

A list of common kernel functions is given in Table~\ref{tab:kernel}.   The first six functions of the table (triangular, Epanechnikov, quadratic, quartic, tricube, spherical, Cauchy) are compactly supported, while the last two functions (exponential, Gaussian) are
infinitely extended but integrable.

\subsection{Kernel Weights} \label{ssec:kernel_weights}
The kernel functions can be used to define weights for the interaction between two points $\bfs_{n}$ and $\bfs_{m}$ as follows

\beq
\label{eq:kernel-weight}
w(\bfr_{n,m}; \bfh) = \frac{K\left(\bfr_{n,m}/h_{n}\right)}{\sum_{k=1}^{N}\sum_{l=1}^{N} K\left(\bfr_{k,l}/h_{k}\right) }, \quad \text{for } \, n,m = 1\, \ldots, N,
\eeq
where $\bfr_{n,m}=\bfs_{n} - \bfs_{m}$ is the distance vector, while $h_{n}$ is a bandwidth specific to the point $\bfs_{n}$ that reflects the sampling density around that point. As a result of the local bandwidth definition in~\eqref{eq:kernel-weight}, the kernel weights are
not necessarily symmetric with respect to interchange of the indices, i.e., in general
$w(\bfr_{n,m}; \bfh) \neq w(\bfr_{m,n}; \bfh)$, since the local sampling density around the point $\bfs_{n}$ can be quite different than the sampling density around $\bfs_{m}$.

In general, one can consider $N$ different bandwidths. However, then the spatial model involves more parameters than data values, since the SLI model has other coefficients in addition to the bandwidths. It is better to define the \emph{local bandwidths} in terms of a small set of free parameters. An intuitive and flexible choice is to use
\beq
\label{eq:bandwidth}
h_{n} = \mu\, D_{n,[k]}(\Samp),
\eeq
where $h_{n}$ is the bandwidth assigned to the point $\bfs_{n}$,  $D_{n,[k]}(\Samp)$ is the  distance between  ${\bfs}_{n}$  and its $k$\emph{-nearest neighbor} in  the sampling set $\Samp$, and $\mu > 0$ is a positive bandwidth tuning parameter. The local distances $D_{n,[k]}(\Samp)$ can be efficiently computed using near-neighbor estimation algorithms, while the global coefficient $\mu$ becomes a model parameter that is determined from the data via the cross-validation procedure as described in Section~\ref{ssec:sli-estim}.  The above definition of the bandwidth implies that larger bandwidth values are assigned to points in areas where the sampling is sparse.

\begin{table}
\caption{\label{tab:kernel} List of  common kernel functions. The normalized distance $u$ is given by $u= \| \bfr \|/h$. $\vartheta(x)=1$ if $x\ge 0$ and $\vartheta(x)=0$ if $x< 0$
is the unit step function. } \centering
\renewcommand\tabcolsep{10pt} % default is 6
\renewcommand\arraystretch{1.2} %default is 1
\begin{tabular}{ll}
\hline%    \hline
\hline
Name &  Equation $(u \ge 0)$ \\
\hline%    \hline\hline
Triangular & $K(u)=(1-u)\,\vartheta(1-u)$  \\[1ex]
Epanechnikov & $K(u)=(1-u)^2 \, \vartheta(1-u)$ \\[1ex]
Quadratic & $K(u)=(1-u^2) \, \vartheta(1-u)$  \\[1ex]
Quartic (biweight) &  $K(u)=(1-u^2)^2 \, \vartheta(1-u)$  \\[1ex]
Tricube  & $K(u)=(1-u^3)^{3} \, \vartheta(1-u)$  \\[1ex]
Spherical &  $K(u)= (1-1.5u+0.5u^3) \, \vartheta(1-u)$  \\[1ex]
Truncated Cauchy & $K(u)=1/(1+u^2) \, \vartheta(1-u)$ \\[1ex]Exponential & $K(u)=\exp(-u)$   \\[1ex]
Gaussian &  $K(u)=\exp(-u^2)$  \\[1ex]
\hline\hline %\hline
\end{tabular}
\end{table}

\subsection{Gradient-based SLI Energy Function}
\label{ssec:gradient-sli}

The SLI model is defined in terms of a Gibbs-Boltzmann exponential joint density of the following form~\cite{Chiles12}

\[
 f(\xsam) = \frac{1}{Z(\bmthe)} \E^{-\HH (\xsam ;\bmthe)},
\]
where $\bmthe$ is the SLI parameter vector that needs to be determined from the data.
The function $\HH (\xsam ;\bmthe)$ is the energy of the  SLI model which incorporates the local interactions. The denominator $Z(\bmthe)$ is known in physics as the \emph{partition function} and represents a normalization constant~\cite{Mussardo10}. It will not be necessary to evaluate this constant in order to estimate the parameters or to predict at non-measured positions.

In the case of scattered data, a  parsimonious SLI energy functional that includes squared fluctuation  and squared ``gradient'' terms is given by~\cite{dth15}

\beq
\label{eq:sli-ene}
\HH (\xsam ;\bmthe)     =
    \frac{1}{2\,\la \, }   \left[ \frac{(\xsam - \bfmx \,)^{\top} (\xsam - \bfmx\,)}{N} +  c_1 \,    {\mathcal{S}_1}(\xsam; \bfh) \right],
\eeq
where $\bfmx$ is the vector of  the local expected values (which is assumed constant herein), $\la$ is a  \emph{scale coefficient} that is proportional to the overall variability of the data, $c_{1}>0$ is a positive \emph{rigidity coefficient} that weighs the cost of gradients in the data, $\mathcal{S}_1(\xsam; \bfh)$ is the kernel-average square gradient, $\bfh = (h_{1}, \ldots, h_{N})^\top$ is a vector of \emph{local kernel bandwidths} that is used to determine the local bandwidths, and
$\bmthe  = \left( \la, \mx, c_{1}, k, \mu   \right)^\top$.

The \emph{gradient-based energy} terms are given by the following kernel average

\beq
\label{eq:sli-gradient}
\mathcal{S}_1(\xsam;\bfh)= \big\langle \left( x_{n} - x_{k}\right)^{2} \big\rangle_{\bfh}
= \sum_{n=1}^{N}\sum_{k=1}^{N} w(\bfr_{n,k}; \bfh)\left( x_{n} - x_{k}\right)^{2},
\eeq

\noindent where $\big\langle \left( x_{n} - x_{k}\right)^{2} \big\rangle_{\bfh} $ represents the Nadaraya-Watson average~\cite{Nadaraya64,Watson64} of the square differences of the data values.

\subsection{SLI Precision Matrix Formulation} \label{ssec:PrecMatForm}

The SLI energy function~\eqref{eq:sli-ene} is a quadratic function of the sample values. Hence, it can be equivalently expressed in the following form that involves the precision matrix ${\mathbf J}(\bmthe)$:

\beq
\label{eq:sli-ene-J}
\HH (\xsam ;\bmthe)  =
\frac{1}{2}
  (\xsam - \bfmx)^{\top} \, {\mathbf J}(\bmthe) \,(\xsam - \bfmx).
\eeq

The precision matrix has the following structure

\begin{subequations}
\label{eq:sli-J}
\beq
\label{eq:sli-J-a}
{\mathbf J}(\bmthe')    =     \frac{1}{\la }
\,\left\{ \frac{\Iden}{N} + c_{1} \, {\mathbf J}_{1}(\bfh)  \right\},
\eeq
and depends on the reduced parameter vector $\bmthe' = \left( \lambda, c_{1}, k, \mu \right)^\top$.  In the above, $\Iden$ is the $N \times N$ identity matrix which results from the energy of the squared fluctuations, and  ${\mathbf J}_{1}(\bfh)$ is  the \emph{gradient precision sub-matrix} which results from the sum of the squared differences and is given by

\beq
\label{eq:sli-J-grad}
[{\mathbf J}_{1}(\bfh)]_{n,m}  = - w_{n,m}(\bfh) - w_{m,n}(\bfh) +
\delta_{n,m} \, \sum_{k=1}^{N} \left[ w_{n,k}(\bfh) + w_{k,n}(\bfh)\right],
\eeq
\end{subequations}

\noindent where $w_{n,m}(\bfh)  = w(\bfr_{n,m}; \bfh)$ is an abbreviation of the kernel weights~\eqref{eq:kernel-weight}  and $\delta_{n,m}=1$ if $m=n$ while $\delta_{n,m}=0$ if $m \neq n$ is the \emph{Kronecker delta}.

\begin{theorem}
The SLI precision matrix defined by~\eqref{eq:sli-J} is strictly positive definite for $\la >0$ and $c_{1}>0$. Moreover, it is diagonally dominant.
\end{theorem}

\begin{proof}
The  matrix ${\mathbf J}(\bmthe')$ is positive definite if it is  real-valued and  symmetric, every non-zero $\bXo=(x_{1}, \ldots, x_{n})^\top$ and every $n \in \mathbb{N}$ it holds that $\bXo^{\top} {\mathbf J}(\bmthe') \bXo \ge 0.$ The matrix ${\mathbf J}(\bmthe')$ defined by~\eqref{eq:sli-J} is real-valued  and symmetric by construction. In addition, if we set $\bXo' = \xsam - \bfmx$, it follows from~\eqref{eq:sli-ene-J} that  ${\bXo'}^{\top} {\mathbf J}(\bmthe') \bXo' = 2\HH(\bXo)$. This equation holds for any sample $\xsam$ and for any $\mx \in \R$, thus for all possible $\bXo$. However, the energy is equivalently expressed based on~\eqref{eq:sli-ene} as follows

\[
2\HH(\bXo; \bmthe)= \frac{1}{\la} \left[ \, \frac{{\bXo'}^\top \bXo'}{N} + c_{1}
\big\langle \left( x_{n} - x_{k}\right)^{2} \big\rangle_{\bfh} \right].
\]
It is straightforward to show based on~\eqref{eq:sli-ene} and~\eqref{eq:sli-gradient} that since $c_{1}, \la >0$, the energy is strictly positive for any non-zero vector $\bXo'$. If all the entries of the vector $\bXo'$ take the same value, the gradient term, which is proportional to $c_{1}$, vanishes; however, the first term is still positive. Hence, the matrix ${\mathbf J}(\bmthe')$ is \emph{strictly positive definite}.

Diagonal dominance requires that for all $n=1, \ldots, N$ the following condition be satisfied

\[
| J_{n,n} | \ge \sum_{m \neq n=1}^{N}  | J_{n,m}|.
\]

The diagonal dominance condition is stricter than positive definiteness. Hence, the former implies the latter (but not vice versa).
From~\eqref{eq:sli-J}  it follows that

\begin{subequations}
\label{eq:sli-J-mn}
\begin{align}
J_{n,n} = & \frac{1}{N \la} + \frac{c_{1}}{\la} \sum_{m \neq n=1}^{N} \left( w_{n,m} + w_{m,n} \right),
\\
J_{n,m} = & -\frac{c_{1}}{\la} \left( w_{n,m} + w_{m,n} \right), \; m \neq n.
\end{align}
\end{subequations}
Based on the above and the non-negativity of the weights, i.e., $w_{n,m} \ge 0$ for all $m,n=1, \ldots, N$, it follows that

\[
| J_{n,n} |- \sum_{m \neq n=1}^{N}  | J_{n,m}| = 1/N\la.
\]

\noindent Hence the diagonal dominance of the SLI precision matrix is proved. 

\end{proof}

\subsection{SLI Prediction}
\label{ssec:sli-predict}
In this section we show how to use the SLI model to estimate the unknown values of the field at the sites $\bfz_{p} \in \Pd$. Let us consider that  one the prediction points $\bfz_{p}$ is added to the sampling network. The SLI energy function is then modified as follows:

\beq
\label{eq:sli-ene-pred}
{\hHH} (\xsam,\xo_{p}  ;{\bmthe})= \HH (\xsam  ;{\bmthe}) + \frac{J_{p,p}\, {\xo'}^{2}_{p}}{2} + \frac{1}{2}\sum_{n=1}^{N} \left(J_{n,p} + J_{p,n} \right)\,  \xo'_{p} \xo'_{n}.
\eeq

The first term in~\eqref{eq:sli-ene-pred} is  the sample energy~\eqref{eq:sli-ene}, while the second and third terms represent the additional energy contribution that involves $\bfz_{p}$: the second term is the self-interaction of the prediction point which contributes proportionally to the square of the fluctuation, and the third term involves the interaction of the prediction point with the sampling points. The latter represents local gradients that reflect differences between the fluctuation at $\bfz_p$ and neighboring sample points.

The energy~\eqref{eq:sli-ene-pred} determines the conditional probability density function at the prediction point, which is given by
\beq
\label{eq:conditional-pdf}
f(\xo_{p} \mid \xsam,\bmthe) \propto \exp\left[ -{\hHH} (\xsam,\xo_{p}  ;{\bmthe})\right].
\eeq

The optimal SLI prediction at $\bfz_{p}$ is equal to the value that minimizes the energy~\eqref{eq:sli-ene-pred} and thus maximizes the conditional pdf. This value is a stationary point of the energy, i.e., it satisfies

\[
\hat{\xo}_{p}= \arg \min_{\xo_p} {{\hHH}}(\xsam, \xo_p ; \bmthe^{\ast}_{-\lambda}),
\]
where $\bmthe^{\ast}_{-\lambda}$ denotes the \emph{optimal parameter vector} (excluding $\lambda$) which is estimated from the data as shown in Section~\ref{ssec:sli-estim}. For now we assume that $\bmthe^{\ast}_{-\lambda}$ is known.

The first term in~\eqref{eq:sli-ene-pred} excludes the prediction point and thus it drops out in the minimization. The first-order derivatives of the next two energy terms lead to a linear equation for the stationary point, which admits the following unique solution

\beq
\label{eq:sli-prediction}
\hat{\xo}_{p}  = \mx - \frac{1}{ J_{p,p}(\bmthe^{\ast})} \, \sum_{n=1}^{N}  \,  J_{p,n}(\bmthe^{\ast}) \, (\xo_{n} - \mx).
\eeq

\noindent In~\eqref{eq:sli-prediction} we used the symmetry of the precision matrix, i.e., the fact that $J_{p,n} = J_{n,p}$. The precision matrix entries that involve the prediction points are obtained from~\eqref{eq:sli-J-mn} as follows

\begin{subequations}
\label{eq:sli-J-p}
\begin{align}
\label{eq:sli-J-pp}
    J_{p,p}(\bmthe^{\ast}) = & \frac{1}{N \la^\ast} + \frac{c_{1}^\ast}{\la^\ast} \sum_{n=1}^{N} \left( w_{n,p} + w_{p,n} \right),
    \\[1ex]
\label{eq:sli-J-pn}
    J_{p, n}(\bmthe^{\ast}) = & -\frac{c_{1}^\ast}{\la^\ast} \left( w_{n,p} + w_{p,n} \right).
\end{align}
\end{subequations}
In light of the kernel weight definition~\eqref{eq:kernel-weight}, the summands in~\eqref{eq:sli-J-pp} are either zero (if the distance between $\bfs_{n}$ and $\bfz_{p}$ exceeds the local bandwidth) or positive (otherwise).

The predictor~\eqref{eq:sli-prediction} is a stationary point of the energy. It remains to  verify that it represents the minimum so that the prediction corresponds to the mode of the conditional pdf~\eqref{eq:conditional-pdf}. The second-order derivative of the energy~\eqref{eq:sli-ene-pred} with respect to $\xo_p$ is given by $J_{p,p}$. According to~\eqref{eq:sli-J-p} it holds that $J_{p,p} >0$ since $\la^{\ast}, c_{1}^{\ast} >0$ and the kernel weights are positive. Hence the predictor indeed corresponds to the global minimum of the energy.

\paragraph{Properties of the SLI predictor}
\begin{enumerate}
\item Note that the prediction equation~\eqref{eq:sli-prediction} is analogous to the equation that determines the conditional mean in GMRFs. This is not surprising since the SLI predictor finds the minimum energy at the prediction point, conditionally on the sample values.

\item The SLI predictor is \emph{unbiased} as can be seen from~\eqref{eq:sli-prediction}: the second term on the right hand side, which is proportional to the fluctuations, vanishes upon calculating the ensemble average.

\item The SLI predictor is not an exact interpolator. According to~\eqref{eq:sli-prediction}, exactitude requires that if $\bfz_{p} \to \bfs_{n^{\ast}} \in \Samp$ then  $J_{p,n} = 0$, for all $ n \neq n^\ast$. This could only occur if the bandwidth tends to zero, or if it is sufficiently small to exclude all other points besides  $\bfs_{n^{\ast}}$ for all $n^{\ast} \in \{1, \ldots, N\}$.  However, the difference between the SLI prediction and the true value is usually negligible. This is due to the fact that the sample value at the nearest point to the prediction location has the largest linear SLI weight and hence the most impact on the prediction.

\item The prediction is independent of the \emph{scale coefficient $\la$}. This becomes obvious upon inspection of~\eqref{eq:sli-prediction}, since the prediction involves the ratio of precision matrix elements.

\item  The analogy of SLI with GMRFs allows us to calculate the \emph{conditional variance} at the prediction point as
\beq\label{eq:conditional-variance}
\sigma^{2}_{p} = \frac{1}{{J}_{p,p}} \propto \la^{\ast}.
\eeq

The conditional variance endows the SLI method with a measure of \emph{prediction uncertainty}  which is analogous to the kriging variance. Just like the kriging variance, the SLI conditional variance does not explicitly depend  on the data values. It only depends indirectly to the extent that the data influence the model parameters.

\end{enumerate}

\paragraph{Computational efficiency:} The numerical cost of the prediction equation~\eqref{eq:sli-prediction} is proportional to $N$. Since the predictor is applied $P$ times to generate the map grid, the numerical complexity is $\Or(N\,P)$.  According to~\eqref{eq:sli-J-p} the numerical complexity of calculating the precision matrix is $\Or(N)$ per prediction point; therefore the overall complexity is again $\Or(N\,P)$. However,  the kernel weights~\eqref{eq:kernel-weight} involve a double summation over the sampling points which has a computation complexity $\Or(N^\beta)$, where $1<\beta \le 2$. The maximum complexity rate, $\beta=2$, is obtained for infinite-support kernels. For compactly supported kernels,  rates  $\beta <2$ can be achieved. In general, the upper bound of the SLI computational complexity is $\Or(N^2+N\,P)$. The memory space required for storage of the precision matrix scales as $\Or(q N^2)$, where $q$ is the \emph{sparsity} of the precision matrix; for compactly supported kernels typically $q\ll 1$.

Kriging has a numerical complexity $\Or(N^3)$ for the inversion of the covariance matrix, and $\Or(N^2+N\,P)$ for the calculation of the kriging weights and the prediction~\cite{Zhong16}.  The kriging complexity is practically dominated by the inversion of the covariance matrix. In addition, kriging requires the storage of the covariance matrix which consumes $\Or(N^2)$ of memory space.  The computational requirements of kriging can be reduced by introducing various modifications, e.g., using search neighborhoods with a small number of sampling points around each prediction point~\cite{Olea12} or by means of tapered covariance functions that have compact supports~\cite{Furrer06}.

\subsection{SLI Parameter Estimation}
\label{ssec:sli-estim}
The  SLI parameter vector is given by $\bmthe  = \left( \la, \mx, c_{1}, k, \mu  \right)^\top$.
In addition to the data, the estimated $\bmthe$ also depends on the kernel function. We use the following approach to estimate the SLI model parameters from the data:

\begin{enumerate}\itemsep0.3em

\item The mean $\mx$ is set equal to the sample average. This is a reasonable choice in the absence of spatial trends, especially for large datasets. However, it is also possible to  consider $\mx$ as a parameter to be determined from the data by means of cross validation as shown by~\citet{dth19}.

\item The order $k$ of the near-neighbor used to estimate the local bandwidths according to~\eqref{eq:bandwidth} is typically assigned an integer  value determined by experimentation with the data. Herein we use $k=3$; we justify this choice below.

\item The optimal rigidity coefficient $c_{1}^{\ast}$ and  global bandwidth parameter $\mu^{\ast}$ are determined by means of  leave-one-out (LOO) cross-validation (CV). In LOO CV sample points are removed one at a time. Their values  are then estimated by means of the remaining $N-1$ sample points and the SLI prediction equation~\eqref{eq:sli-prediction}. After all the points are removed and estimated,  a specified \emph{validation measure} (e.g., mean absolute error, root mean square error) that quantifies the discrepancy between the estimates and the true values is  calculated. The validation measure is used as the \emph{cost function} for the optimization of the SLI parameters.

\item The optimal value of the scale coefficient is determined by means of~\cite{dth15cg}
\beq
\label{eq:lambda}
\lambda^{\ast}    = \frac{2\HH(\xsam;\bmthe^{\ast}_{-\la})}{N}, \; \mbox{where} \;
\bmthe^{\ast}_{-\la}= (1, \, \mx^\ast, \, c_{1}^\ast, \, k^{\ast}, \, \mu^\ast  )^\top.
\eeq
The value of $\la^\ast$ is only used in  determining the conditional variance but not for prediction.
\end{enumerate}

The above procedure is repeated for different choices of the kernel function, and the kernel model which optimizes the cost function is selected.

Maximum likelihood estimation (MLE) is also possible within the SLI framework. However, in using MLE we are forced to estimate the partition function, $Z(\bmthe)$, which involves calculating the determinant of the large, albeit sparse, precision matrix ${\mathbf J}(\bmthe')$.

%%%%%%%%%%%%%%%%%%%%%%%%%%%%%%%%%%%%%
\section{Description of the Data and Study Area} \label{Sec:Data}

We apply the SLI model to a set of coal thickness data from 11416 drill-holes. The data come from the Gillette field in Campbell County, Wyoming (USA). The locations of the drill-holes are shown in Fig.~\ref{fig:geotrhseis}. Campbell County lies within the Powder River Basin and includes  two of the largest coal mines in the world: North Antelope Rochelle coal mine and Black Thunder coal mine. The Gillette coal field  contains significant quantities of exploitable coal resources that are low in sulfur and ash content.

\begin{figure}[h!]
\centering
\includegraphics[width=.5\textwidth]{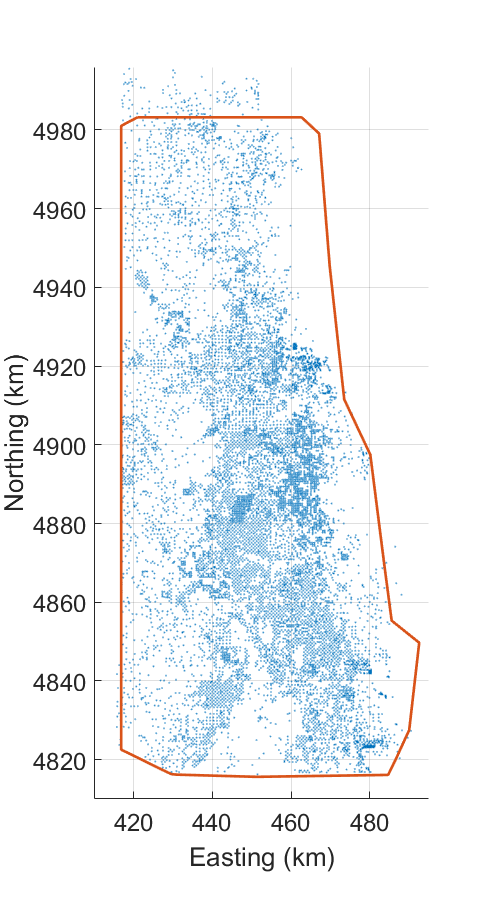}
\caption{Drill-hole locations of the coal thickness data from the Gillette field. The red line represents the border of the area within which the resources will be estimated.}
\label{fig:geotrhseis}
\end{figure}

The study domain extends over an area that covers approximately 9500 km\textsuperscript{2}. We focus on coal beds from the upper part of  Fort Union Formation which include the Roland bed and  coal in the Wyodak-Anderson zone.
The stratigraphic cross-section showcasing the Wyodak-Anderson coal zone for part of the Gillette Field is shown in Fig.~\ref{fig:TomhCampbell}.
The geological formations present in the Gillette field are shown in Fig.~\ref{fig:Coalzon}~\cite{Ellis99}.
Coal from the Wyodak zone has higher mean gross calorific value (4760 kcal/kg), lower mean ash yield (5.8\%), and lower mean total sulfur content (0.46\%) compared with the respective values for the combined Wyodak-Anderson coal zones~\cite{Ellis02}.

\begin{figure} [h!]
\centering
 \includegraphics[scale=0.5]{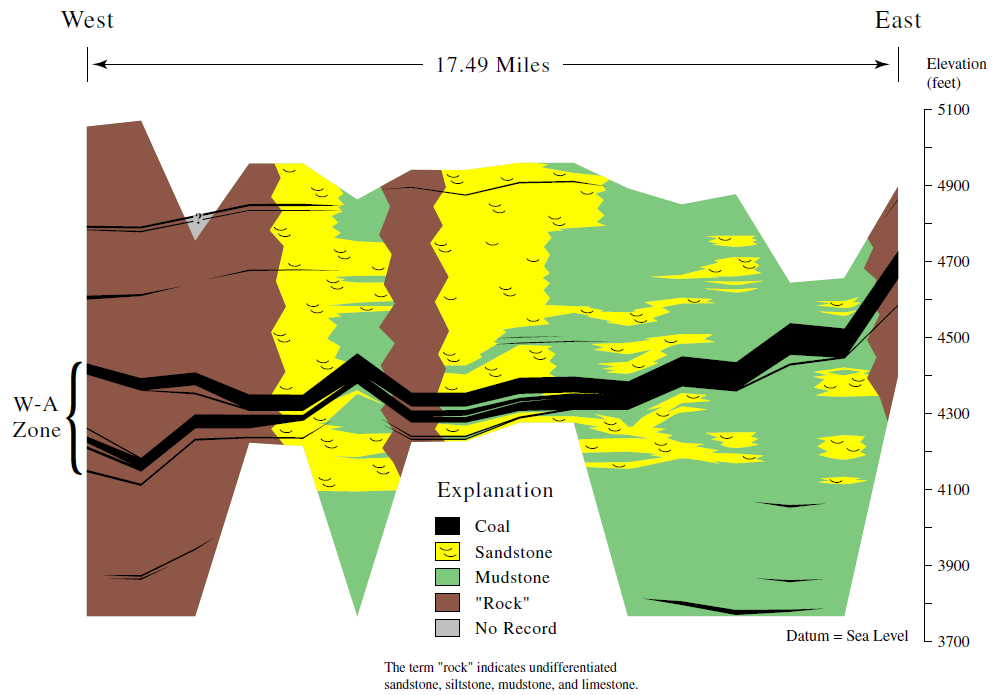}
\caption{Structural cross section showing the Wyodak-Anderson coal zone (W-A Zone) of the Gillette coalfield~\cite{Stricker07}. }
\label{fig:TomhCampbell}
\end{figure}

% \subsection{Gillette Coal Field} \label{ssec:Gilette coal}

\begin{figure} [h!]
\centering
 \includegraphics[scale=0.5]{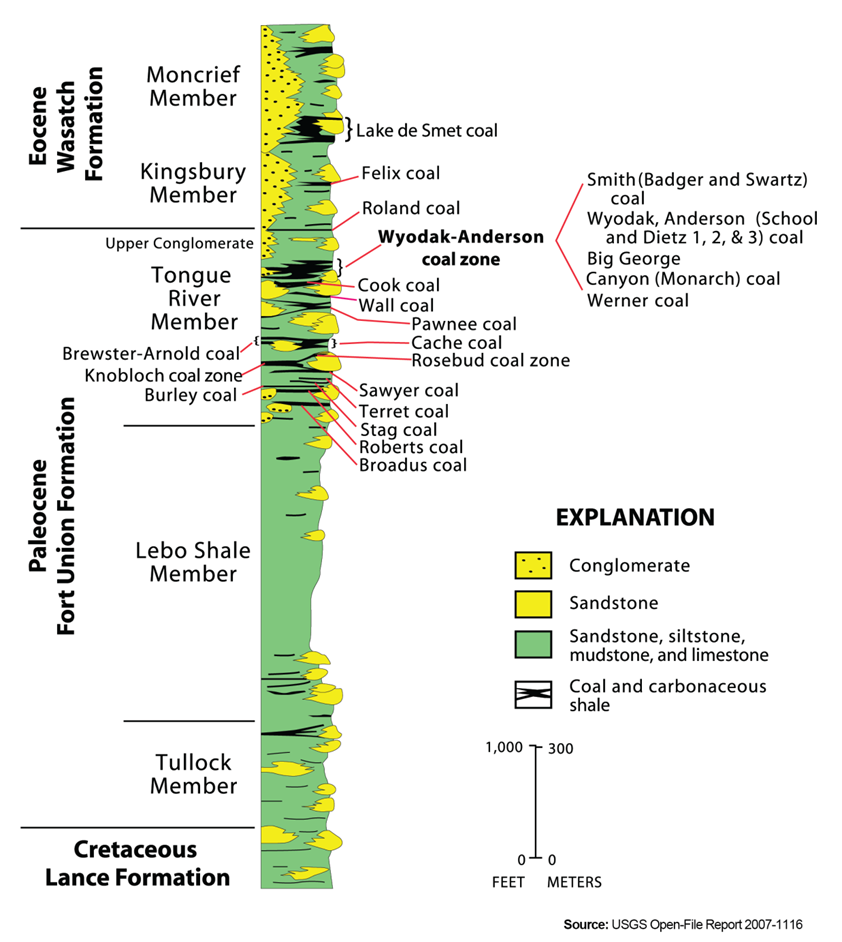}
\caption{Section showcasing geological formations and coal zones for the Gilette Field~\cite{Stricker07}. }
\label{fig:Coalzon}
\end{figure}

The main  statistics of the marginal distribution of the coal thickness data are listed in Table~\ref{Tab:Data Stat}. The frequency histogram which represents the empirical  probability distribution of the thickness values is shown in Fig.~\ref{fig:histogramdata}. Both the  statistics shown in Table~\ref{Tab:Data Stat} and the histogram in  Fig.~\ref{fig:histogramdata} show evidence of mild deviation from the normal (Gaussian) distribution.  The spatial distribution of the drill-holes, and hence the sampling density, varies significantly in space as evidenced in Fig.~\ref{fig:geotrhseis}. The median of the nearest distance between drill-holes is 458\,m.

\begin{table}
\caption{Statistics of coal thickness  in the upper part of  Fort Union Formation based on data from the 11416 drill-hole locations of Gillette Field.}
\centering
\renewcommand\tabcolsep{9pt} % default is 6
\renewcommand\arraystretch{1.2} %default is 1
\begin{tabular}{cccccc}
\hline
\hline
Mean (m) & Median (m) & Maximum (m) & Minimum (m) & Standard deviation (m) & Skewness \\
\hline
7.83 & 7.71 & 22.67 & 0.33 & 2.99 & 0.13 \\
\hline\hline
\end{tabular}
\label{Tab:Data Stat}
\end{table}

\begin{figure}[!ht]
\centering
\includegraphics[scale=0.85]{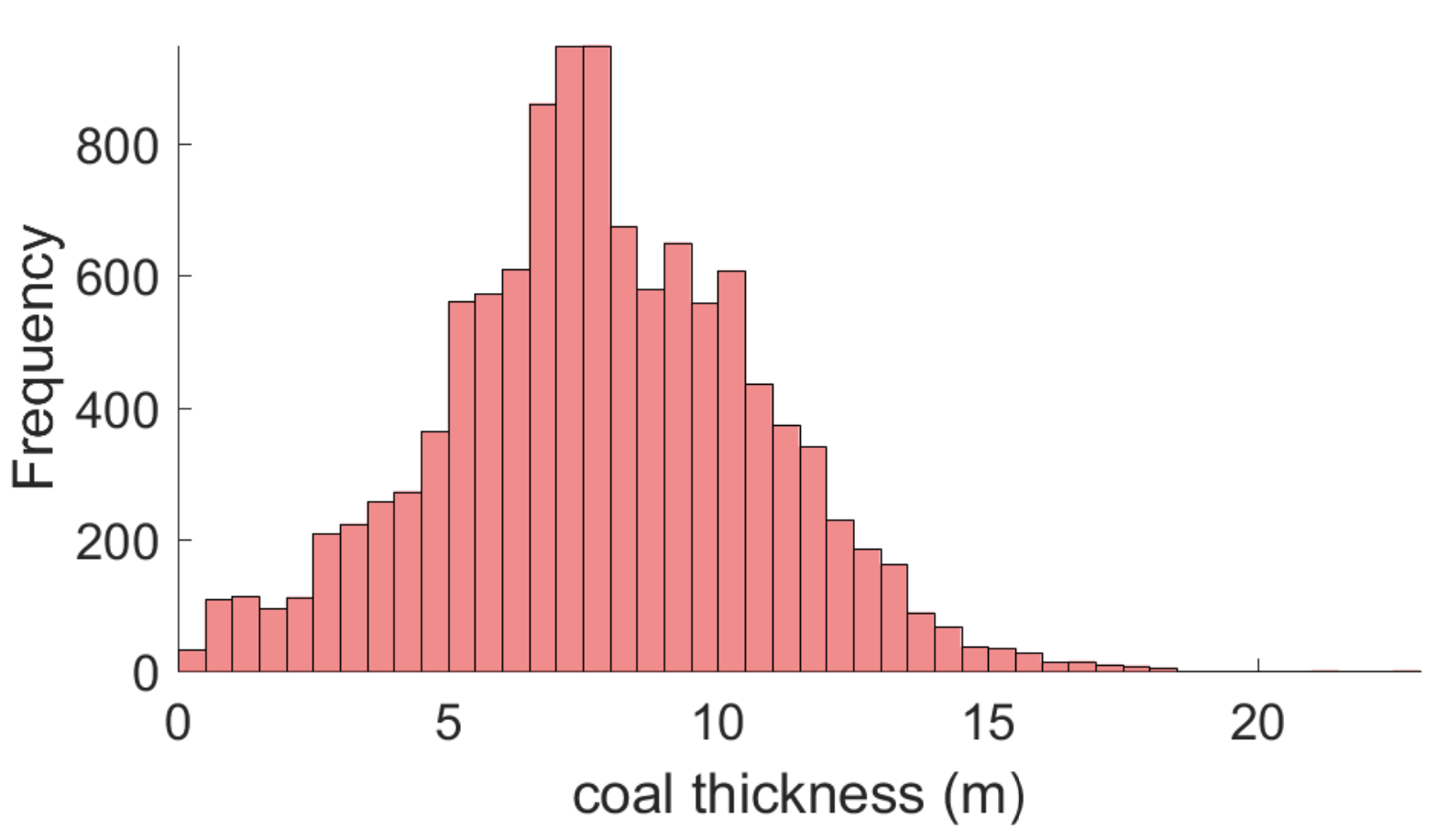}
\caption{Frequency histogram of the coal thickness distribution  in the upper part of Fort Union Formation based on data from  11416 drill holes in Gilette field.}
\label{fig:histogramdata}
\end{figure}

%Grid: A coarse grid of $240 \times 100$ cells, with each cell being 758.0\,m $\times$ 753.6\,m.

All the numerical tests below are conducted and timed on a desktop computer running Windows 10 operating system equipped with Intel(R) Core(TM) i7-6700K CPU $@$ 4.00GHz and 16.0 GB RAM installed. The numerical implementation was performed in the  MATLAB programming environment (release R2016a).
%%%%%%%%%%%%%%%%%%%%%%%%%%%%%%%%%%%%%

\section{SLI Model Estimation for Gillette Field}
\label{sec:SLI-estim-gilette}
To use SLI for interpolation we first need to estimate the model parameters  from the Gillette field coal data. We perform this task using LOO-CV as described in Section~\ref{ssec:sli-estim}.
Kernel weighting functions are utilized in the SLI predictions as explained in Section~\ref{ssec:KernelFun} to  determine the  weights given by equation~\eqref{eq:kernel-weight}.
To find the optimal kernel function we use LOO-CV  on the entire dataset and select the kernel function that optimizes the CV cost function. Herein we use the mean absolute error (MAE) as the cost function.  The optimization is performed in the Matlab environment using the constrained minimization function \verb+fmincon+ with the interior-point algorithm~\cite{Nocedall0}.

\subsection{Kernel function selection}
We compare the predictive performance of different compactly supported kernels from the list shown in Table~\ref{tab:kernel}. Compact kernels lead to sparse precision matrices (see Section~\ref{ssec:PrecMatForm}) that are computationally less intensive for big datasets.
The comparison is based on the mean absolute error obtained by means of leave-one-out cross validation. The SLI predictions are based on the equation~\eqref{eq:sli-prediction}.

Table~\ref{table:campbell_kernels_MAE_k=3} shows the cross-validation MAE for the different kernel functions obtained using neighbor order $k=3$ to define the distances $D_{n,[k]}(\Samp)$ in equation~\eqref{eq:bandwidth}. As is evidenced Table~\ref{table:campbell_kernels_MAE_k=3}, the parameter $c_{1}$ that controls the gradient precision sub-matrix~\eqref{eq:sli-J-a} takes values close to 115 for all the kernels. The bandwidth control parameter $\mu$ shows more dependence on the kernel function. The MAE takes very similar values for all the studied kernels, ranging between $1.12$ m and $1.13$ m (truncating at two decimal points).
Based on these results we select the spherical kernel which gives one of the lowest MAE (slightly less than $1.12$ m). The Epanechnikov kernel also returns practically the same MAE value, but a higher value for $\mu$; as the latter determines the bandwidth, higher values of $\mu$ imply increased memory requirements and computational time. In general, we expect that any of the kernels listed in Table~\ref{table:campbell_kernels_MAE_k=3} (except for the uniform kernel which has a slightly higher MAE) will lead to comparable performance.

\begin{table}[!ht]
\caption{Comparison of the mean absolute error (MAE) value in coal thickness estimation by SLI using different kernel functions with neighbor order $k=3$ and initial parameter values $c_{1}^{(0)}=115$ and $\mu^{(0)}=1.5$.  MAE is optimized by means of a constrained optimization method. The value of $\mu$ is constrained in the interval $[0.5, 5]$, and $c_1 \in [\mathrm{eps}, \infty]$, where eps~$\approx 2.22 \times 10^{-16}$ is the floating-point relative accuracy.}
\label{table:campbell_kernels_MAE_k=3}
\centering
\renewcommand\tabcolsep{10pt} % default is 6
\renewcommand\arraystretch{1.2} %default is 1
%		\resizebox{14.5cm}{!}{
\begin{tabular}{l|cccc}
\hline
% \multicolumn{5}{c}{K=3; Param0\_s=[115 1.5]}       \\
% \hline % In-table horizontal line
\hline
Kernel & $c_1$ & $\mu$ & $\lambda$ & MAE (m) \\
\hline
Uniform   & $115.0000$                 & $1.5000$              & $0.1949$               & $1.1315$            \\
Triangular   & $115.0003$            & $2.0018$           & $0.1802$               & $1.1213$            \\
Quadratic   & $115.0003$            & $1.8958$           & $0.1907$               & $1.1228$            \\
Biweight   & $115.0002$            & $2.1772$           & $0.1859$               & $1.1209$            \\
Tricubic   & $115.0002$            & $2.1885$           & $0.1891$               & $1.1215$            \\
Epanechnikov    & $115.0003$  & $2.4526$  & $0.1676$  & $1.1199$    \\
\textbf{Spherical}   & $\mathbf{115.0003}$            & $\mathbf{2.3675}$           & $\mathbf{0.1727}$               & $\mathbf{1.1199}$            \\
\hline\hline
\end{tabular}
\end{table}

\begin{table}[!ht]
\caption{Comparison of the mean absolute error (MAE) value in coal thickness estimation by SLI using the spherical kernel function with neighbor order $k=2, \, 3, \, 4$ and initial parameter values $c_{1}^{(0)}=115$ and $\mu^{(0)}=1.5$. The value of $\mu$ is constrained in the interval $[0.5, 5]$, and $c_1 \in [\mathrm{eps}, \infty]$.  The computational time was between 5 and 9 minutes.}
\label{table:campbell_spherical_MAE_different_k}
\centering
\renewcommand\tabcolsep{10pt} % default is 6
\renewcommand\arraystretch{1.2} %default is 1
%		\resizebox{14.5cm}{!}{
\begin{tabular}{l|cccc}
\hline
% \multicolumn{5}{c}{K=3; Param0\_s=[115 1.5]}       \\
% \hline % In-table horizontal line
\hline
Neighbor order & $c_1$ & $\mu$ & $\lambda$ & MAE (m) \\
\hline
 $k=$ 2   & 115.0003            & 2.6795           & 0.1671              & $\mathbf{1.1182}$             \\
 $k=$ 3  & 115.0003    & 2.3675    & 0.1727    & 1.1199    \\
 $k=$ 4    & 115.0002  & 2.0518  & 0.1717  & 1.1188    \\
 \hline\hline
\end{tabular}
\end{table}

\subsection{Distribution of SLI bandwidths}
The distribution of bandwidths for the spherical kernel  is shown in the histogram plot of Fig.~\ref{fig:Bandwidths}. The bandwidths are obtained from~\eqref{eq:bandwidth} using $k=3$ and $\mu = 2.3675$.
As evidenced in the histogram plot, the bandwidth distribution is dispersed with a peak at $\approx 1500$~m; however, there are also values as large as $8000$~m. The dispersion of the distribution reflects the variations of the sampling density (cf. Fig.~\ref{fig:geotrhseis}) in Gillette field: Larger (smaller) values of the bandwidth are obtained in areas of sparser (denser) sampling.
The choice of neighbor order $k=3$ is based on  Table~\ref{table:campbell_spherical_MAE_different_k}, which compares the parameter estimates and the MAE values obtained for three different neighbor orders: $k=2, 3, 4$.  The table shows that $k=2$ gives the lowest MAE; however, all the MAE are equal within the second decimal place. In order to be impartial in our selection we choose $k=3$.  
As the neighbor order $k$ increases, the optimal $\mu$  decreases (see Table~\ref{table:campbell_spherical_MAE_different_k}), thus tending to keep the value of the optimal bandwidth constant. This behavior further demonstrates the ability of SLI to adjust the bandwidth adaptively without user input.

\begin{figure}[h!]
\centering
\includegraphics[scale=0.65]{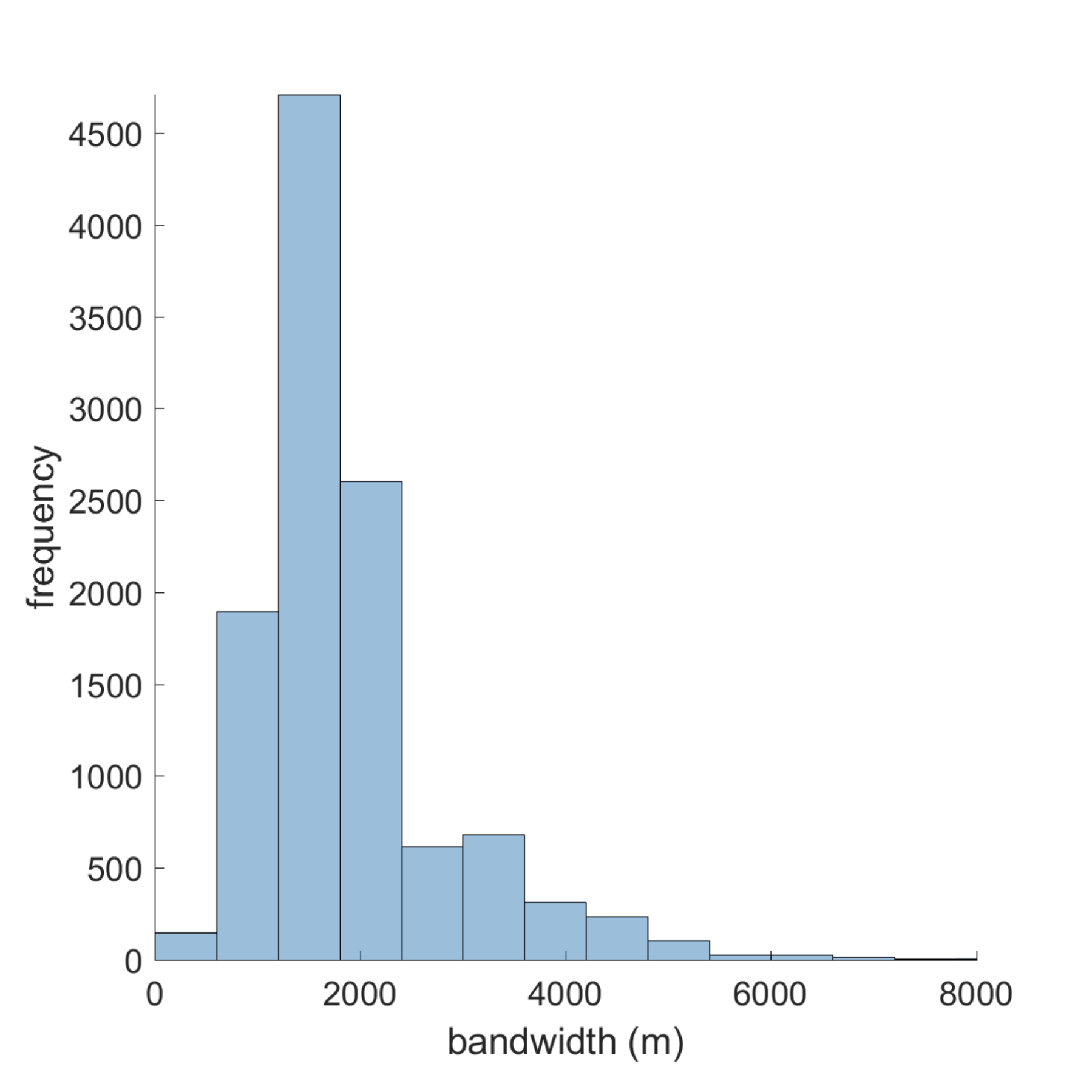}
% Histogram_bandwidths_meters_multiplied_by_mu
\caption{Frequency histogram of the kernel bandwidths (m) used in the SLI method. The bandwidths were determined from~\eqref{eq:bandwidth} using the spherical kernel with near-neighbor order $k=3$.}
\label{fig:Bandwidths}
\end{figure}
%  They are the actual bandwidth, they have already been multiplied by $\mu$.

\subsection{Parameter estimation using global optimization}
\label{ssec:global}
Note that the estimated $c_{1}^\ast=115.0003$ is very close to the initial guess $c_{1}^{(0)}=115$. This result suggests that the optimization, which is performed by means of the constrained optimization Matlab function \verb+fmincon+ with an interior-point algorithm, may be stuck in a local minimum. To investigate this possibility we use the global search optimization function in Matlab, which can find global solutions to problems that contain multiple maxima or minima. This procedure is time intensive, but it allows for a better exploration of the SLI parameter space.
We use the Matlab \verb+GlobalSearch+ solver which generates a sequence of starting points (initial guesses) in combination with the interior-point local optimization routine~\cite{Ugray07}. The global optimizer is run with different initial values and  bounds for $c_{1}$.

The results of global optimization for four different combinations of  initial values and bounds for $c_{1}$ are reported in Table~\ref{table:campbell_spherical_MAE_k=3_global}.
All runs take several hours to complete, due to the exhaustive search of the parameter space by the global optimizer. As evidenced in Table~\ref{table:campbell_spherical_MAE_k=3_global} the global optimizer gives different solutions depending on the initial conditions. These solutions are
 in a different region of parameter space  than the local optimum $(c_{1}^{\ast} \approx 115, \mu^{\ast}\approx 2.37, \la^{\ast} \approx  0.17)$. However, all the solutions lead to very similar results for the cost function (i.e., the MAE). This investigation shows that different local minima of the SLI cost function are essentially identical with respect to prediction performance. We further discuss this issue in Section~\ref{sec:cross-validation}.  Note that both the global and local solutions give essentially the same estimate for $\mu^\ast$.
At the same time, the ratio $c_{1}^{\ast}/\la^{\ast}$ is also approximately the same for both the global ($\approx$ 670) and the local ($\approx$ 676) estimates.
This observation is further analyzed and justified in Section~\ref{ssec:sli-stable}.

\begin{table}[!ht]
\caption{Estimates of SLI parameters using LOO-CV and optimal mean absolute error (MAE) for the coal thickness data. The estimates are derived using the Matlab global optimization solver GlobalSearch with different initial  guesses and constraint intervals  for $c_{1}$. The SLI model is implemented using the spherical kernel function with neighbor order $k=3$.  The parameter $\mu$ is constrained in the interval $[0.5, 5]$ and its initial guess is $\mu^{(0)}=1.5$.} %The initial guesses for the parameter values are  $c_{1}^{(0)}=115$ and $\mu^{(0)}=1.5$.}
\label{table:campbell_spherical_MAE_k=3_global}
\centering
\renewcommand\tabcolsep{10pt} % default is 6
\renewcommand\arraystretch{1.3} %default is 1
%		\resizebox{14.5cm}{!}{
\begin{tabular}{l|ccccc}
\hline
% \multicolumn{5}{c}{K=3; Param0\_s=[115 1.5]}       \\
% \hline % In-table horizontal line
\hline
Constraints and Initialization & $c_{1}^{\ast}$ & $\mu^{\ast}$ & $\lambda^{\ast}$ & MAE (m) & $c_{1}^{\ast}/\lambda^{\ast}$\\
\hline
$c_1 \in$ [eps, $\infty$], \, $c_{1}^{(0)}=115$   & $1.0001 \times 10^{4}$            & $2.3742$           & $14.9147$      & $1.1996$  &    670.48     \\
$c_1 \in [10, 10^{8}]$,  \, $c_{1}^{(0)}=115$     & 7.3092 $\times 10^{7}$            & 2.3743           & 1.0900 $\times 10^{5}$      & $1.1196$      &    669.72  \\
$c_1 \in$ [eps, $\infty$], \, $c_{1}^{(0)}=10$   &  8.3285 $\times 10^{13}$           &   2.3762         &  1.2427 $\times 10^{11}$               &     1.1196  & 669.91      \\
$c_1 \in [10, 10^{8}]$, \, $c_{1}^{(0)}=10$    &  8.2394 $\times 10^{7}$           &   2.3743        &  1.2287 $\times 10^{5}$             &     1.1196   & 670.58     \\
\hline\hline
\end{tabular}
\end{table}

\section{Variogram Modeling}
\label{sec:OrdKrig}
In this section we estimate the variogram model that will be used for ordinary kriging interpolation and resources estimation. We use the robust variogram estimator~\cite{Cressie1980} for the empirical variogram of coal thickness which we then fit to five common variogram models.

The  robust variogram estimator  is given by

\begin{equation}
\label{equation:variogram}
\hat{\gamma}(\bfr)= \dfrac{\large\left[ \dfrac{1}{2N({\bf{r}})} \sum_{i=1}^{N(\bfr)} \left| x(\bfs_{i})- x(\bfs_{i}+\bfr) \right|^{1/2}\large\right]^4 } {0.457+\dfrac{0.494}{N(\bfr)}+\dfrac{0.045}{N^2(\bfr)}} ,
\end{equation}

\noindent where $\bfr$ is the distance (lag) vector between two points and $N(\bfr)$ is the number of pairs separated by  $\bfr$.

We used~\eqref{equation:variogram} to calculate the omnidirectional variogram. We then tested five different theoretical variogram models (Spherical, Cubic, Gaussian, Power, and Exponential) combined with a nugget term. The fits of the theoretical models $\gamma(\bfr;\bmthe)$ to the empirical variogram are presented in Fig.~\ref{fig:models}. The best fitting model was selected by minimizing the sum of weighted squared errors given by~\eqref{eq:error_vario}. The sum of weighted squared errors for the five variogram models tested are shown in Table~\ref{table:variofit_errors}. The spherical model is shown to minimize the sum of weighted squared errors. The optimal  parameters for the spherical variogram model are shown in Table~\ref{table:parameters_variogram}.

\begin{equation}
\label{eq:error_vario}
\varepsilon(\bmthe)= \sum_{i=1}^L \frac{ \left[\gamma(\bfr_{i};\bmthe)-\hat{\gamma}(\bfr_{i})\right]^2}{\gamma^2(\bfr_{i};\bmthe)}
\end{equation}

 \begin{table}[!ht]
\caption{Sum of weighted squared errors between the empirical variogram and the five theoretical models that are tested based on equation~\eqref{eq:error_vario} (units are in m\textsuperscript{4}).}
\centering
\renewcommand\tabcolsep{10pt} % default is 6
\renewcommand\arraystretch{1.2} %default is 1
{\begin{tabular}{cccccc}
\hline\hline
Variogram Model & Gaussian & Cubic & \bf{Spherical} & Power & Exponential\\[1ex]
\hline
Fitting Error $\varepsilon(\bmthe)$ & 11.55 & 2.67& {\bf{1.92}} & 6.79&   3.37 \\[1ex]
\hline\hline
\end{tabular} }
\label{table:variofit_errors}
\end{table}

\begin{figure}[!ht]
\centering
\includegraphics[width=.8\textwidth]{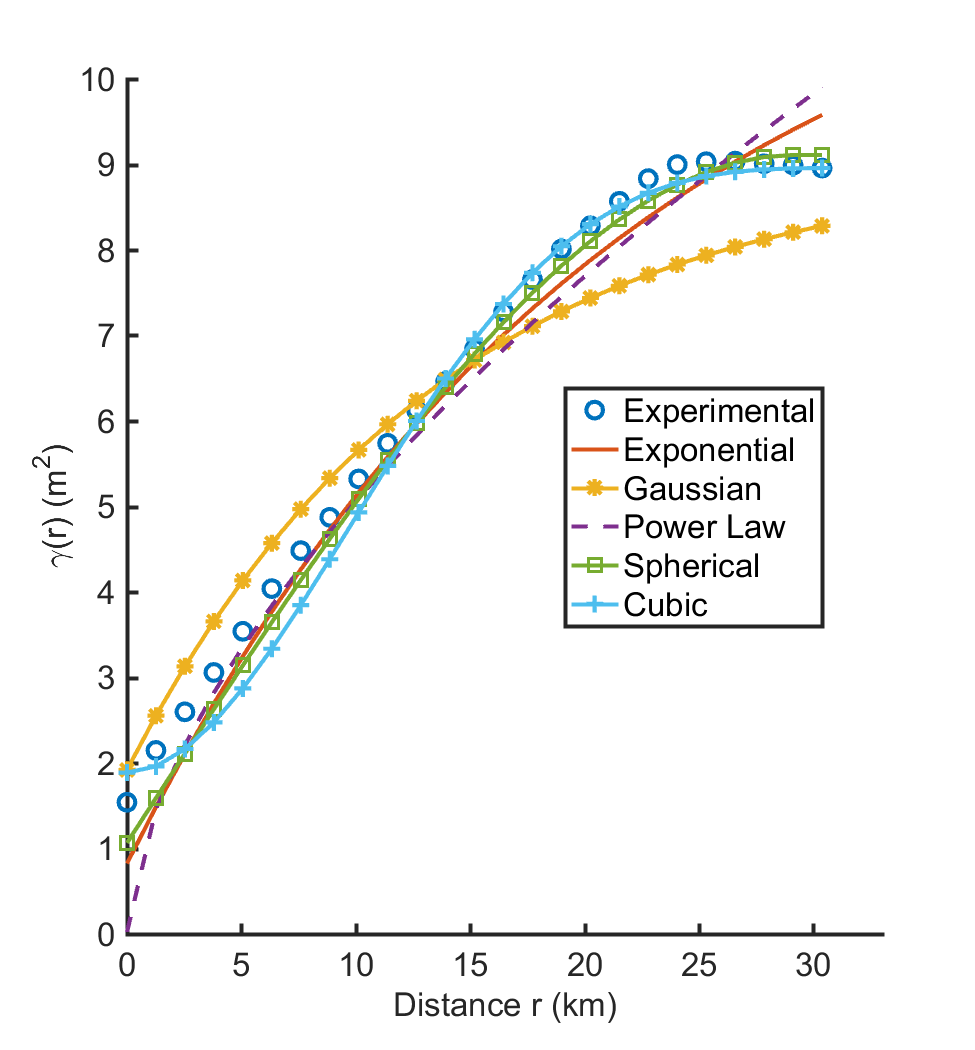}
\caption{Robust variogram estimator (circles) and optimal fits to theoretical models (Gaussian, Spherical, Power, Exponential, and Cubic). The horizontal axis is the lag distance $\| \bfr \|$.  The vertical axis represents the variogram values of coal thickness.}
\label{fig:models}
\end{figure}

\begin{table}[!ht]
\caption{Parameters of the optimal spherical variogram model: $\sigma^2$ is the correlated variance, $\xi$ is the correlation length, and $c_0$ is the nugget variance. The parameters are estimated by minimizing the error function~\eqref{eq:error_vario}.}
\centering
\renewcommand\tabcolsep{10pt} % default is 6
\renewcommand\arraystretch{1.2} %default is 1
{\begin{tabular} {ccc}
\hline\hline
$\sigma^2$ (m\textsuperscript{2}) & $\xi$ (km) & $c_{0}$ (m\textsuperscript{2})\\[1ex]
\hline
8.05& 29.1&1.07 \\[1ex]
\hline\hline
\end{tabular} }
\label{table:parameters_variogram}
\end{table}

The optimality of the spherical function  for both  the variogram model and the SLI  spherical kernel  is coincidental.
% For one, the Epanechnikov kernel is just as good as the spherical for the SLI model.
The spherical variogram model implies a spherical covariance, while the spherical kernel in SLI enters in the precision matrix.   Finally, we used the omnidirectional variogram estimator and did not consider anisotropic variogram models. We estimated geometric anisotropy using  the method of covariance Hessian identity~\cite{dth02,dth08}. This led to an estimate of the correlation length ratio along the two principal directions  slightly less than two, which is considered as marginally isotropic.

\section{Estimation of Coal Resources}
\label{sec:estimation-coal}

In this section we estimate the coal resources of Gillette field using SLI and ordinary kriging.
Our goal is to compare the SLI predictions with those of ordinary kriging. This comparison allows  assessing the interpolation performance of SLI against an established stochastic method.
The estimate of the resources is based on estimated coal thickness values at the nodes of an interpolation grid.
The interpolation domain (contained inside the red-line border in Fig.~\ref{fig:geotrhseis}) is covered by a map grid comprising $661355$ orthogonal (approximately square) cells. The cell size is 114.7~m $\times$ 134.7~m.

\subsection{SLI Estimation of Coal Resources} \label{ssec:SLI_Gilette_Pred}

The coal deposits in the upper part of Fort Union Formation in  Gillette field are estimated by SLI interpolation at a total volume equal to

\beq
V = A_{c} \, \sum_{p=1}^{P} \hat{x}(\bfz_{p}) = 74.1\times 10^{9}~\mathrm{m}\textsuperscript{3},
\eeq
where $A_{c}$ is the surface area of the interpolation grid cell and $\{ \hat{x}(\bfz_{p}) \}_{p=1}^{P}$ the estimated coal thickness for the layers described in Section~\ref{Sec:Data} at the center points of the grid. These estimates are based on  the spherical kernel function and the model parameters of Table~\ref{table:campbell_kernels_MAE_k=3}.

The SLI method estimates both the most probable coal thickness at each point and the respective prediction variance. The coal thickness prediction map is shown in Fig.~\ref{sfig:SLI_map}. The standard deviation of the SLI prediction,  i.e., the square root of the conditional variance of the SLI predictor given by~\eqref{eq:conditional-variance}, is shown in Fig.~\ref{sfig:s_SLI}.

\begin{figure}[ht!]
\centering
 \begin{subfigure}[b]{0.49\textwidth}
\includegraphics[width=\linewidth]{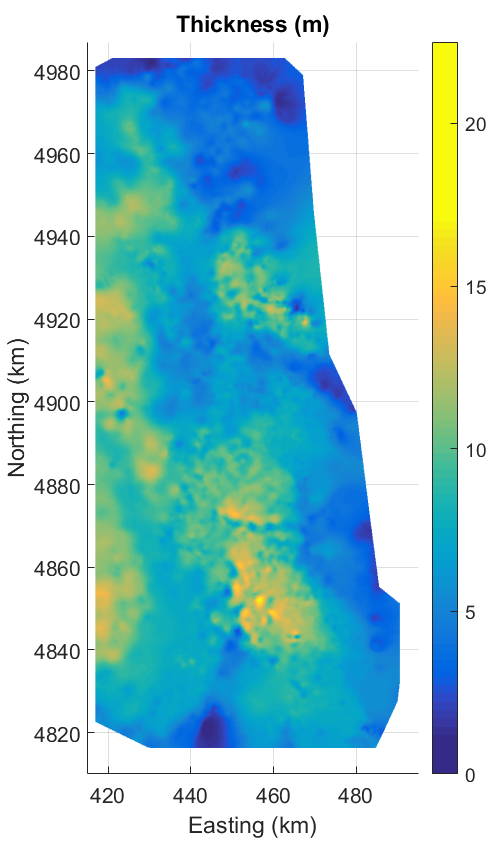}
\caption{SLI Map}
\label{sfig:SLI_map}
\end{subfigure}
 \begin{subfigure}[b]{0.49\textwidth}
\includegraphics[width=\linewidth]{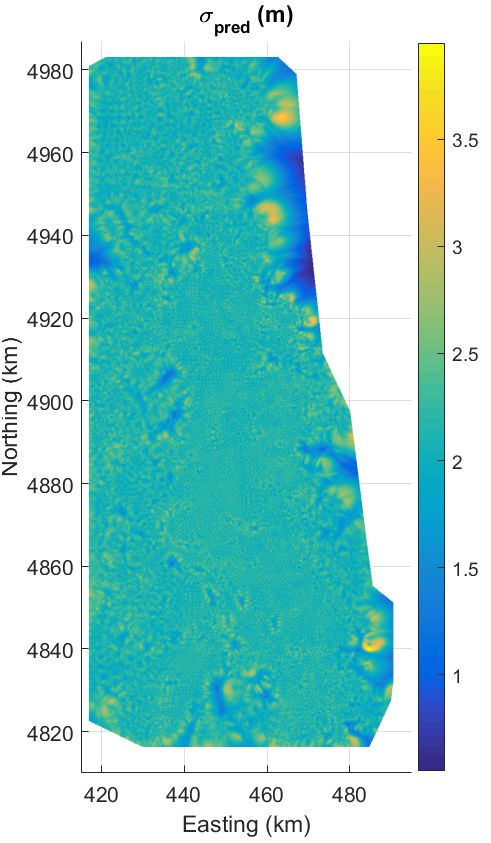}
\caption{SLI StD}
 \label{sfig:s_SLI}
\end{subfigure}
\caption{SLI estimates of coal thickness and standard deviation $\sigma_{\mathrm{SLI}}$ for the upper part of Fort Union Formation in  Gillette  field.}
\label{fig:SLI_map}
\end{figure}

An interesting feature of the SLI standard deviation map is the low uncertainty values near the Northeastern border of the study area which is a region of low sampling density. However, as we show in Fig.~\ref{fig:SLI_map-neighbors}, both the bandwidth and the number of SLI neighbors (i.e., sampling points that are coupled with the prediction point via the SLI model) are rather high (i.e., more than 600) near the boundary. Hence, the uncertainty in this area is constrained by a large, albeit distant, number of neighbors. The bandwidth for each prediction point is determined via the optimization process based on the neighbor order, the kernel function, and the parameter $\mu$. The bandwidth values can thus vary considerably. This is in contrast with fixed-radius kriging, which maintains a constant search radius for all the prediction points. The number of SLI neighbors for each prediction point is determined by the local bandwidth and the sampling point configuration.  An inspection of the SLI standard deviation map (Fig.~\ref{sfig:s_SLI}) and the map of local neighbors (Fig.~\ref{sfig:SLI_map_neighbors}) shows that the number of neighbors is correlated with the SLI standard deviation.  The SLI uncertainty seems to behave more like the standard deviation derived from several stochastic simulations, in which the uncertainty is  proportional both to the data spacing and to the complexity in the fluctuations, than like the kriging standard deviation.

\begin{figure}[h!]
\centering
\begin{subfigure}[b]{0.49\textwidth}
\includegraphics[width=\linewidth]{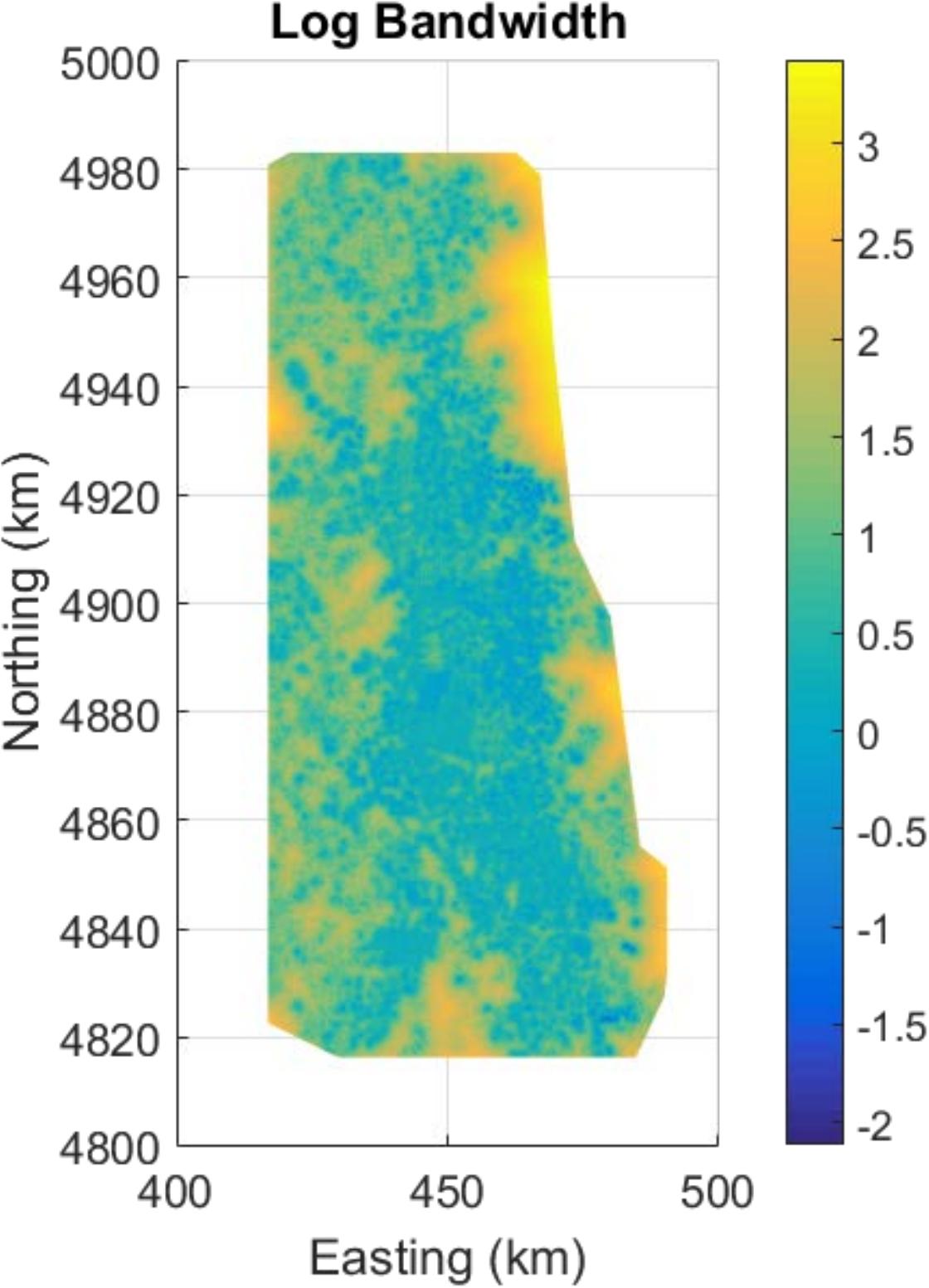}
\caption{Log bandwidth}
\label{sfig:s_SLI}
\end{subfigure}
\begin{subfigure}[b]{0.49\textwidth}
\includegraphics[width=\linewidth]{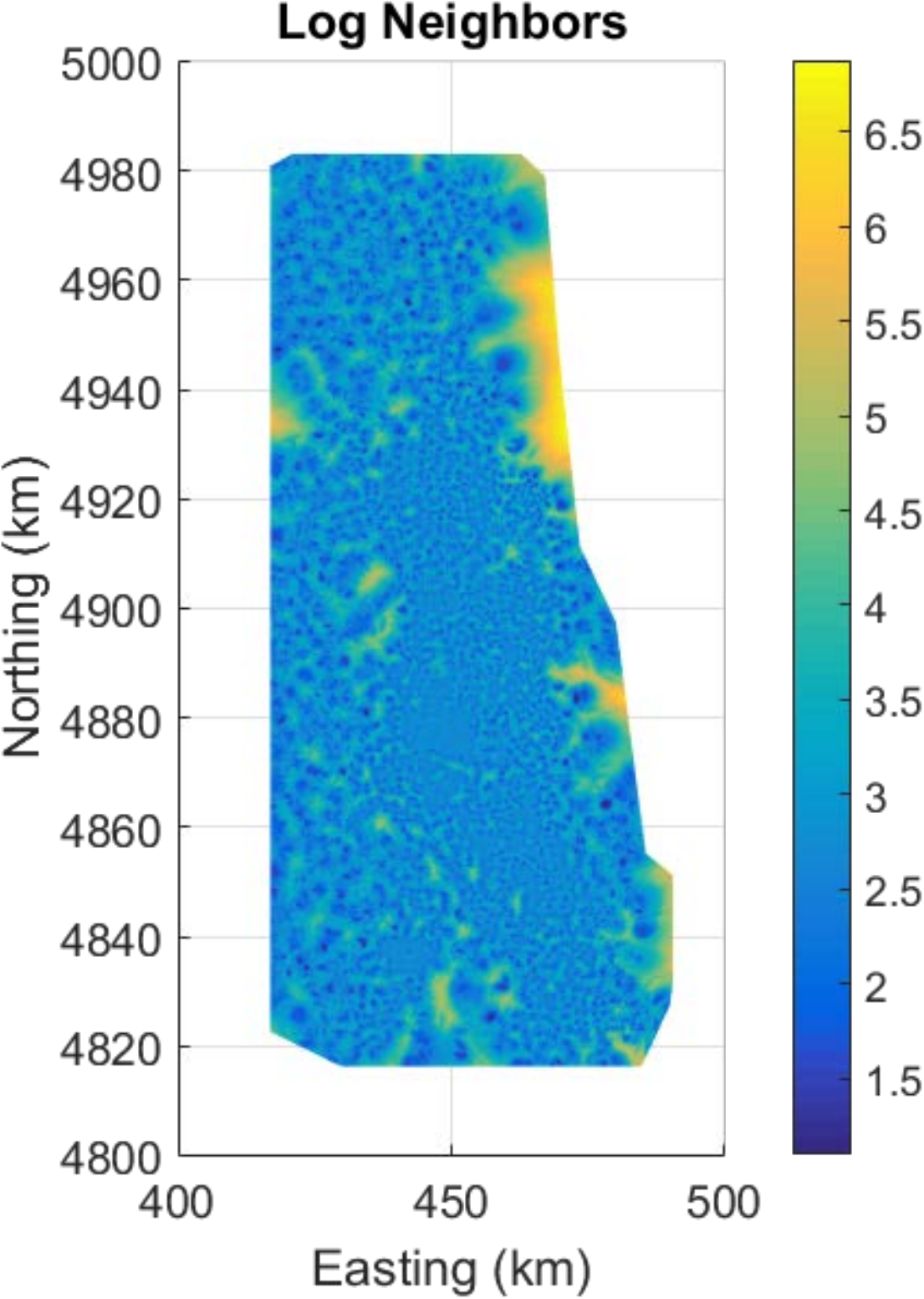}
\caption{Local neighbors}
\label{sfig:SLI_map_neighbors}
\end{subfigure}
\caption{(Left): Logarithm of the SLI bandwidth over the map area. (Right): Logarithm of number of neighbors per map grid point.}
\label{fig:SLI_map-neighbors}
\end{figure}

\subsection{OK Estimation of Coal Resources}
For the estimation of the coal resources in the study area and the visualization of the spatial distribution of coal thickness
we use an  OK interpolation grid identical with the SLI grid.
We tested four different circular search neighborhoods. The respective search radii are equal to 1.5 km, 5.7 km, 9.1 km and 14.6 km. These values were guided by the correlation length of the selected spherical model ($\approx 29$ km). We implemented LOO-CV for each search radius, using the same variogram parameters (shown in Table~\ref{table:parameters_variogram}).

Results for cross validation measures obtained with the search radii of 9.1~km and 14.6~km are shown in Table~\ref{tab:loocvneigh_kriging}.
There is no significant difference between the CV measures obtained with these radii and those obtained with radii of 1.5~km and 5.7~km.
For the purpose of mapping the coal thickness
we use the radius of 14.6~km because it allows almost full coverage of the map grid.
The resulting kriging maps of coal thickness and its standard deviation are shown in Figure~\ref{fig:krigingmaps}.
Note that the colorbar scale next to the OK standard deviation extends to values less than 1~m, i.e., lower than the nugget standard deviation ($\approx 1$~m). This is due to using common color scales for the OK and SLI maps. The values of the OK standard deviation are actually larger than the nugget standard deviation.

\begin{figure} [!ht]
\centering
 \begin{subfigure}[OK prediction]{0.49\textwidth}
\includegraphics[width=\linewidth]{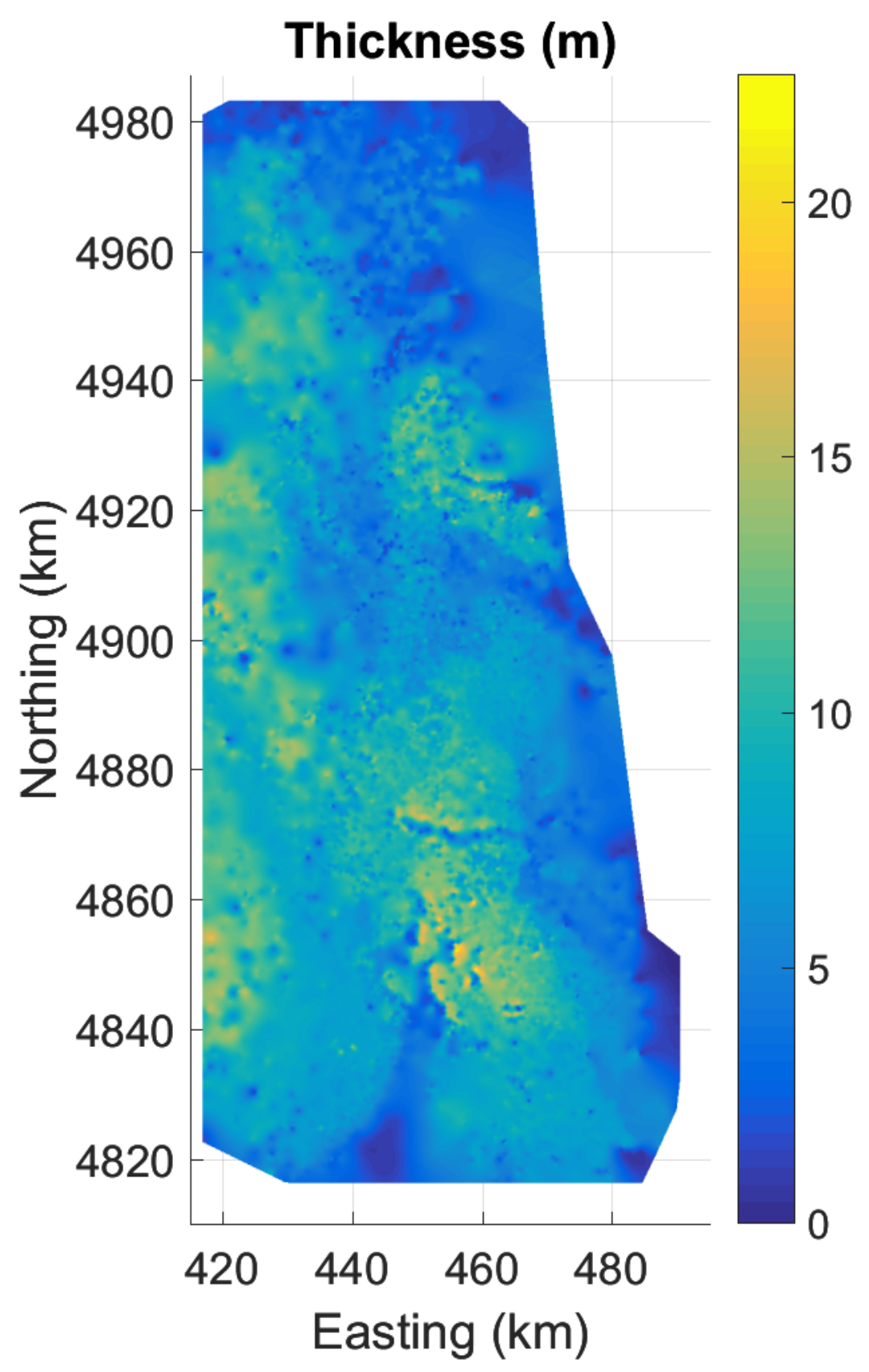}
\label{fig:krigingmap}
\end{subfigure}
 \begin{subfigure}[OK Std]{0.49\textwidth}
\includegraphics[width=\linewidth]{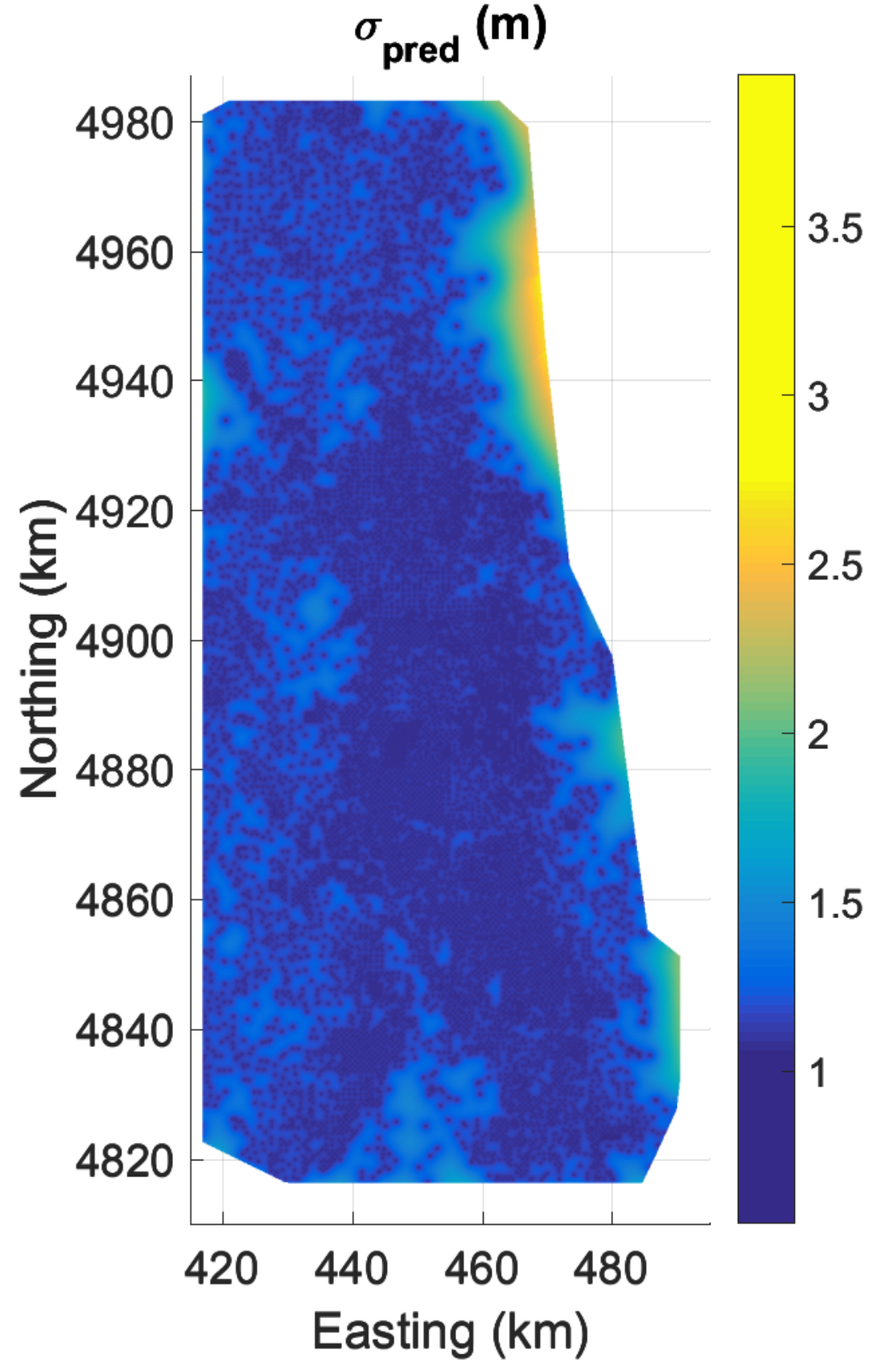}
\label{fig:stdvmap}
\end{subfigure}
\caption{(a) Map of coal thickness based on ordinary kriging. (b) Map of OK standard deviation $\sigma_{\mathrm{OK}}$. Both maps are obtained using the spherical variogram model with the optimal parameters (see Table~\ref{table:parameters_variogram}) and a search neighborhood of 14.6~km.}
\label{fig:krigingmaps}
\end{figure}

Based on ordinary kriging the coal deposits are estimated at  72.4$\times 10^9$~m\textsuperscript{3}, using the spherical variogram model and model parameters of Table~\ref{table:parameters_variogram} with neighborhood 14.6~km.
The SLI prediction  takes $\approx$~12 min, while kriging  with a search radius of 9.1~km (14.6~km) takes $\approx$~42.7~min ($\approx$~319.2~min).

\subsection{Method Comparisons}
In this section we compare the performance of SLI and OK interpolation in terms of (i) the LOO-CV error of the coal thickness estimates (ii) various cross validation statistics and (iii) the interpolated maps of coal thickness.

First,  we focus on the  leave-one-out cross validation error which is given by
\[
\epsilon(\bfs_{i}) = x(\bfs_{-i}) - \hat{x}_{i}(\bfs_{i}; \bmthe^{\ast}_{i}), \; i=1, \ldots, N,
\]
where $\hat{x}_{-i}(\bfs_{i}; \bmthe^{\ast}_{i}) $ is the SLI (OK) estimate at the point $\bfs_{i}$ based on the remaining $N-1$ points and the parameter vector estimate derived from the $N-1$ points.

The histograms of the LOO-CV errors
are shown in Fig.~\ref{fig:Pred_Error} and exhibit overall good agreement. The error distributions are symmetric, reflecting the unbiasedness of the SLI and OK predictors. The vast majority of the error values are distributed between $-$5~m and 5~m. However, there are a few error values with magnitudes between 5m and 10~m as well as some values with larger magnitude.

Figure~\ref{fig:Pred_Error} also exhibits box plots of the LOO-CV errors. The central mark of the box plots is the median of the empirical distribution,
and the edges of the box correspond to the 25th and 75th percentiles. The whiskers extend to the most extreme values that are  not considered outliers.
Values are considered as outliers if they lie  outside the range $[\, q_{3}-1.5(q_{3}-q_{1}), \; q_{3} + 1.5\,(q_{3}-q_{1})\, ]$, where  $q_{1}$ and $q_{3}$ are the 25th and 75th percentiles of the sample data.  This range corresponds to the  99.3\% probability interval for the  normal distribution. Outliers are plotted as individual points (marked by crosses).  A comparison of the two box plots shows that the SLI error is less disperse, while the OK error has more weight in the tails of the distribution, including a couple of values with magnitudes above 15. %\tr{and below $-$15}

\begin{figure}[!ht]
\centering
\includegraphics[width=0.75\textwidth]{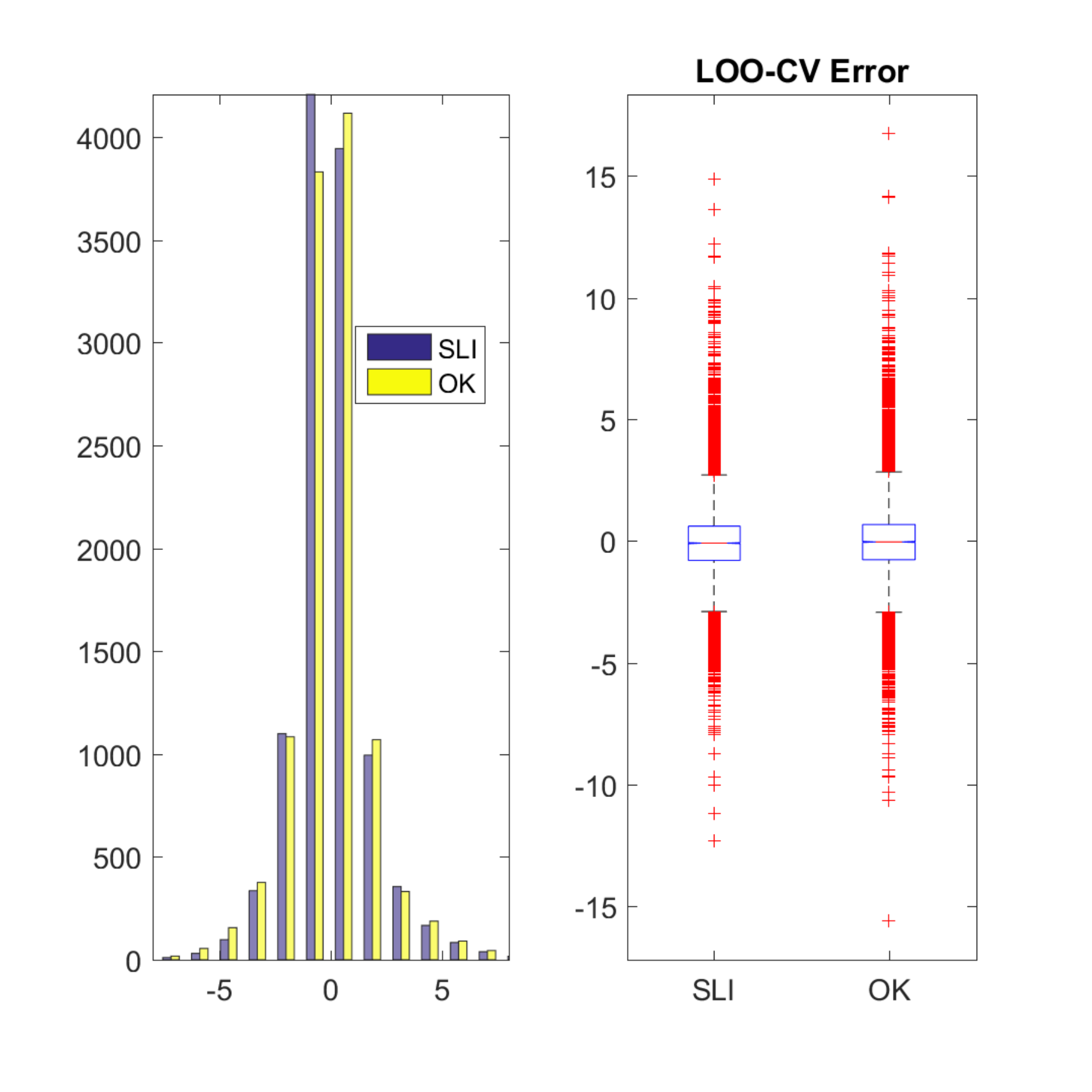}
\caption{Left: Histograms of the SLI (blue online) and OK (yellow online) leave-one-out cross validation errors.  The SLI parameters are given in  Table~\ref{table:campbell_kernels_MAE_k=3}. The spherical variogram parameters used in OK are given in  Table~\ref{table:parameters_variogram}.A kriging search radius of 9.1~km is used. Right: Box plots of the SLI and OK errors (see text for definition of the box plot elements). }
 \label{fig:Pred_Error}
\end{figure}

Next, we compare the predictive performance of SLI and OK in Table~\ref{tab:loocvneigh_kriging} by means of cross-validation statistics.
The following cross-validation measures are considered:
\begin{enumerate}
\item Mean Error (ME) to represent bias.
\item Mean Absolute Error (MAE).
\item Root Mean Square Error (RMSE).
\item Pearson's linear correlation coefficient ($R$).
\item Maximum Absolute Error (MaxAE).
\item Mean Absolute Relative Error (MARE).
\item Root Mean Square Relative Error (RMSRE).
\end{enumerate}

The SLI estimates have lower mean absolute error, root mean square error, maximum absolute error and higher correlation coefficient than the kriging estimates. On the other hand, kriging shows less bias, and lower relative MAE and RMSE.

\begin{table}[!ht]
\centering
\renewcommand\tabcolsep{3pt} % default is 6
\renewcommand\arraystretch{1.2} %default is 1
\caption{LOO-CV measures  of coal thickness based on (i) ordinary kriging with the spherical variogram model and (ii) SLI with spherical kernel and neighbor order $k=3$. The variogram parameters are given in  Table~\ref{table:parameters_variogram} and the SLI parameters are given in  Table~\ref{table:campbell_kernels_MAE_k=3}. A kriging search radius equal to 9.1~km is used. The following validation measures are shown: MAE: mean absolute error;
MaxAE: Maximum  absolute error; MARE: mean absolute relative error; RMSE: root mean square error;  RMSRE: root mean square relative error;    $R$: Pearson's correlation coefficient. S.R.: Search radius.}
\begin{tabular}{l|ccccccc}
\hline\hline
Method & ME (m) & MAE (m) & MARE & RMSE (m) & RMSRE & MaxAE (m) & $R (\%)$  \\[1ex]
\hline\hline
OK (9.1~km S.R.) &  4.0224 $\times  10^{-4}$ & 1.1788 & 0.2480 & 1.8067 & 0.8312 & 16.7286 & 80.27\%  \\[1ex]
OK (14.6~km S.R.) & 4.0573 $\times  10^{-4}$  & 1.1787 & 0.2480 & 1.8066 & 0.8313 & 16.7290 & 80.27\% \\[1ex]
SLI ($k=3$) & 5.6420 $\times 10^{-4}$ & 1.1199 & 0.2480 & 1.7033 & 0.8527 & 14.8781 & 82.23\% \\[1ex]
\hline\hline
\end{tabular}
\label{tab:loocvneigh_kriging}
\end{table}

Finally, the difference between the SLI and OK estimates of coal thickness over the map grid is illustrated in Fig.~\ref{fig:differences_maps}. The four  intervals used are defined as follows: OK$>$SLI corresponds to differences less than $-1$m (higher OK estimates), OK=SLI includes differences between $-1$m and 1m, SLI$>$OK contains differences between 1 and 3.5m, and SLI$\gg$OK  differences higher than 3.5m.  The difference between the SLI and OK predictions ranges from $-7.6$~m (higher OK estimates) to 9.67~m (higher SLI estimates).  Figure~\ref{fig:Difference_pred_map} shows that the SLI estimate of coal thickness at most map grid locations is approximately equal to OK estimates. At locations where this is not true, SLI estimates tend to exceed the OK estimates more frequently than the opposite. This leads to the overall higher SLI-based estimate of resources  ($74.1\times 10^{9}~\mathrm{m}\textsuperscript{3}$ versus $72.4\times 10^{9}~\mathrm{m}\textsuperscript{3}$ based on OK).

Regarding the spatial distribution of prediction uncertainty, Fig.~\ref{fig:Difference_stdv_map} shows that the SLI standard deviation in most places exceeds the OK standard deviation. This is not surprising, since the OK predictor is constructed based on the minimization of the mean error variance while the SLI method does not constrain the variance. At the same time, there are spatial pockets where the SLI uncertainty is lower than the respective OK value. A juxtaposition of the plots in  Fig.~\ref{fig:Difference_stdv_map} and  Fig.~\ref{fig:SLI_map-neighbors} shows that areas with lower SLI variance (e.g., near the northeastern border of the study domain) coincide with areas that have high bandwidth and, especially, high number of SLI neighbors. The presence of many neighbors  tends to constrain the uncertainty, since the precision matrix elements $J_{p,p}$ include a higher number of positive terms $w_{n,p}$ and $w_{p,n}$, according to~\eqref{eq:sli-J-pp}; this tends to increase $J_{p,p}$ and thus, according to~\eqref{eq:conditional-variance}, reduce the variance.

\begin{figure} [!ht]
\centering
 \begin{subfigure}[b]{0.49\textwidth}
\includegraphics[width=\linewidth]{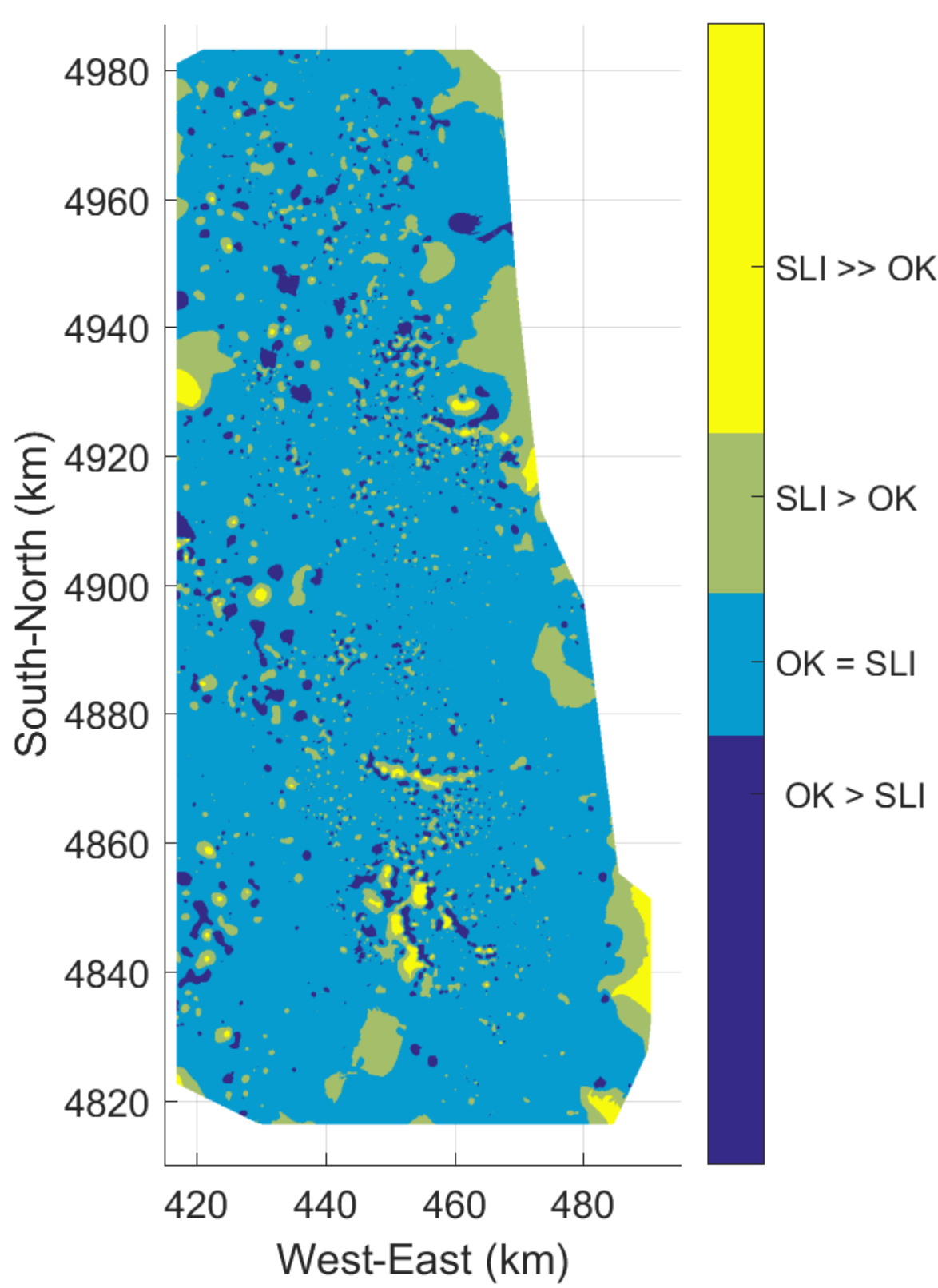}
\caption{Prediction differences}
\label{fig:Difference_pred_map}
 \end{subfigure}
 \begin{subfigure}[b]{0.49\textwidth}
\includegraphics[width=\linewidth]{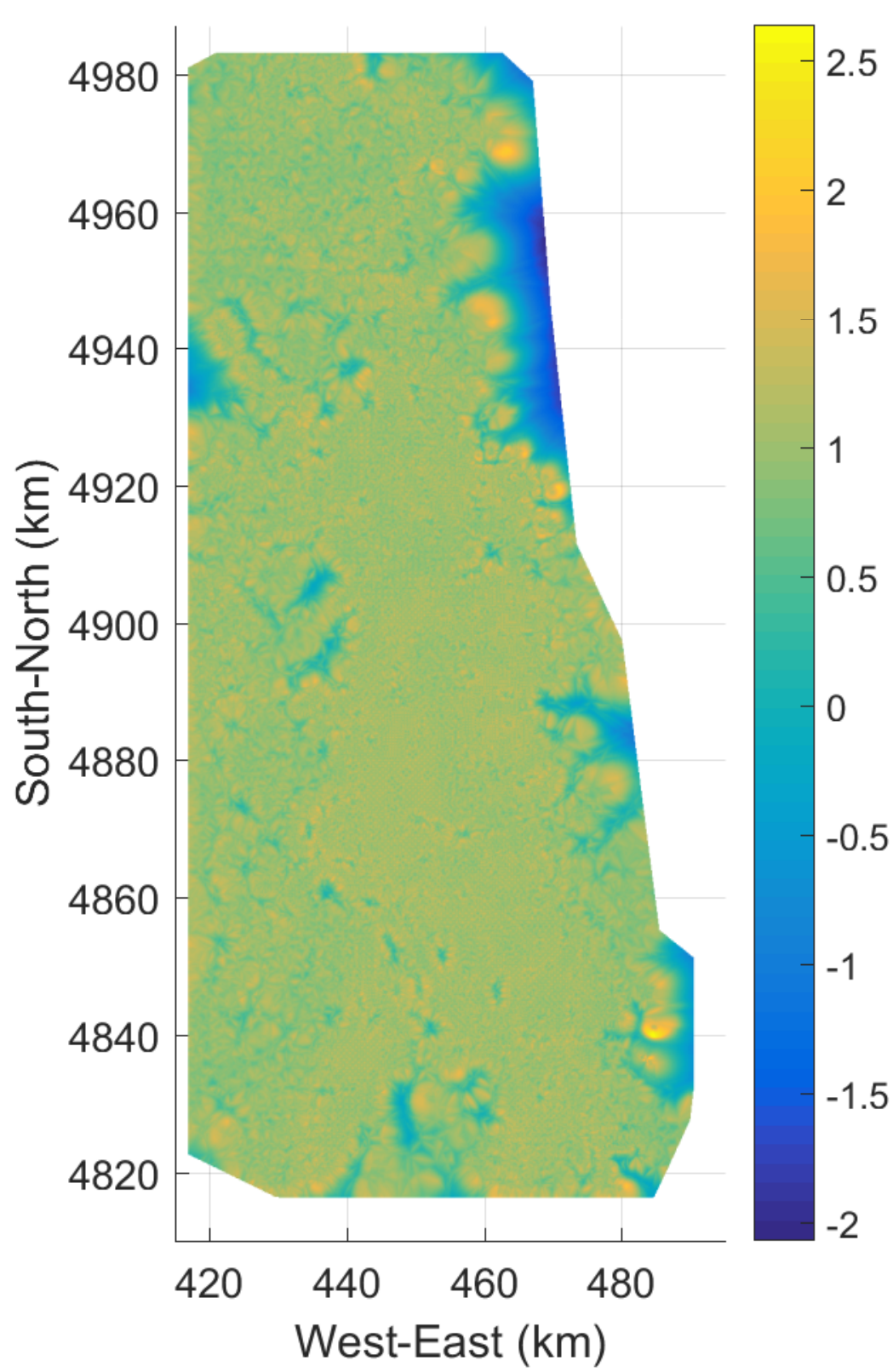}
\caption{Std differences}
\label{fig:Difference_stdv_map}
 \end{subfigure}
\caption{(a) Map of coal thickness differences between the SLI and OK estimates of coal thickness (m). (b) Map of the standard deviation differences ($\sigma_{\mathrm{SLI}}-\sigma_{\mathrm{OK}}$) of coal thickness. SLI maps are obtained using the spherical kernel and neighbor order $k=3$ (see Table \ref{table:campbell_spherical_MAE_different_k}). Kriging maps are obtained using the spherical variogram model with the optimal parameters (see Table~\ref{table:parameters_variogram}) and a search neighborhood of 14.6~km.}
\label{fig:differences_maps}
\end{figure}

\section{Cross Validation Analysis}
\label{sec:cross-validation}
In this section we use the method of cross validation~\cite{Arlot10,Chiles12} with two goals in mind: (i) to investigate the stability of the SLI model parameter estimates and (ii) to better assess the predictive performance of SLI and Kriging.

In cross validation studies the dataset is split in two disjoint sets: the test set  is used to estimate the model parameters, while the training set includes the points where the model performance is tested.  The performance of the model is quantified by calculating a validation statistic (see below) that measures the difference between the data values in the test set and those estimated by means of the model. The split of the dataset is performed multiple times resulting in  different configurations for the test and training sets. This leads to a set of validation statistics which can be averaged over all cross-validation configurations, thus providing a statistical measure of the model's performance.

Herein we use \emph{leave-P-out cross validation (LPO-CV)}, where $ 1 \le P <N$ is an integer number which determines how many points $N_{\mathrm{test}}=P$ are contained in the test set and how many points  $N_{\mathrm{train}}=N-P$ are contained in the training set.
Since the Campbell dataset includes a large number of data points, we use  $P= N_{\mathrm{test}}= \lfloor N \times p \rfloor$, where $\lfloor x \rfloor$ is the floor function which returns the greatest integer that does not exceed $x$, and $p=10\%, 90\%$ is a rate that determines which percentage of the data are used to test the model. In the current study we use 100 repetitions (folds) for a \emph{small training set with $p=10\%$}  (i.e., we keep ten percent of the data in the test set) and with a \emph{large training set with $p=90\%$}.  The test and training sets are randomly selected and they are used in  cross validation analysis with both kriging and SLI.

\subsection{SLI Parameter Stability Analysis}
\label{ssec:sli-stable}

We investigate the dependence of the SLI parameters  on the number and the configuration of the sampling points. We are motivated by the observations in Section~\ref{ssec:global},   multiple local optima with very similar values of the cross-validation cost function seem to exist.
The parameter $\mu$ is not further investigated since its estimate is stable (cf. Section~\ref{ssec:global}).  The SLI parameter estimates are based on the spherical kernel function with neighbor order $k=3$. The initial guesses for the parameter values are $c_{1}^{(0)}=115$ and $\mu^{(0)}=1.5$. The (local) constrained minimization algorithm is used.

The SLI parameters $c_{1}$ and $\lambda$ estimated for each of the 100 folds based on a small training set  (1142 locations)   are presented in Fig.~\ref{fig:train_10_SLI}.   As evidenced in Fig.~\ref{c1}, 61  estimates of $c_1$ have values in the range of $115.0-115.4$. However, there are also two estimates of $c_1$ that are of order  $10$, two estimates in the range $10^2-10^3$, 34 estimates of $c_1$ that are of order  $10^5-10^7$. While this behavior  may seem odd in the first place,   the ratio of $c_1 / \lambda $ is approximately constant, i.e.,
 $c_1 \approx 59.79 \lambda$ as we show in Fig.~\ref{c1-la}.
Note that similar behavior was observed in LOO-CV (see Section~\ref{ssec:global}), albeit the value of the ratio was different ($\approx 670)$.
The value of the ratio in general depends on the configuration of the sampling points (cf.  Table~\ref{table:campbell_spherical_MAE_k=3_global}).

\begin{figure}[ht!]
\centering
\begin{subfigure}[b]{0.49\textwidth}
\includegraphics[width=\linewidth]{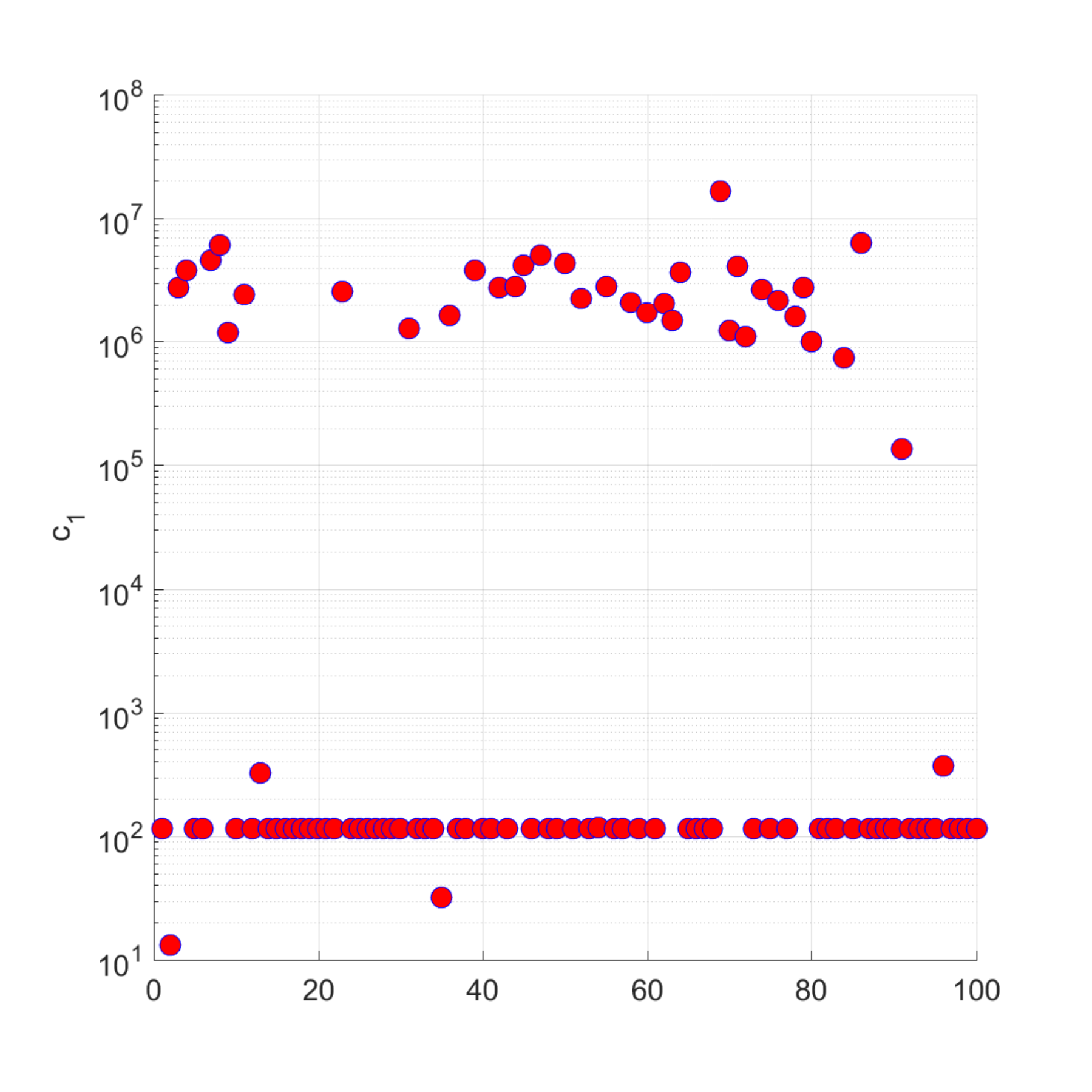}
\caption{$c_{1}$}
\label{c1}
 \end{subfigure}
\begin{subfigure}[b]{0.49\textwidth}
\includegraphics[width=\linewidth]{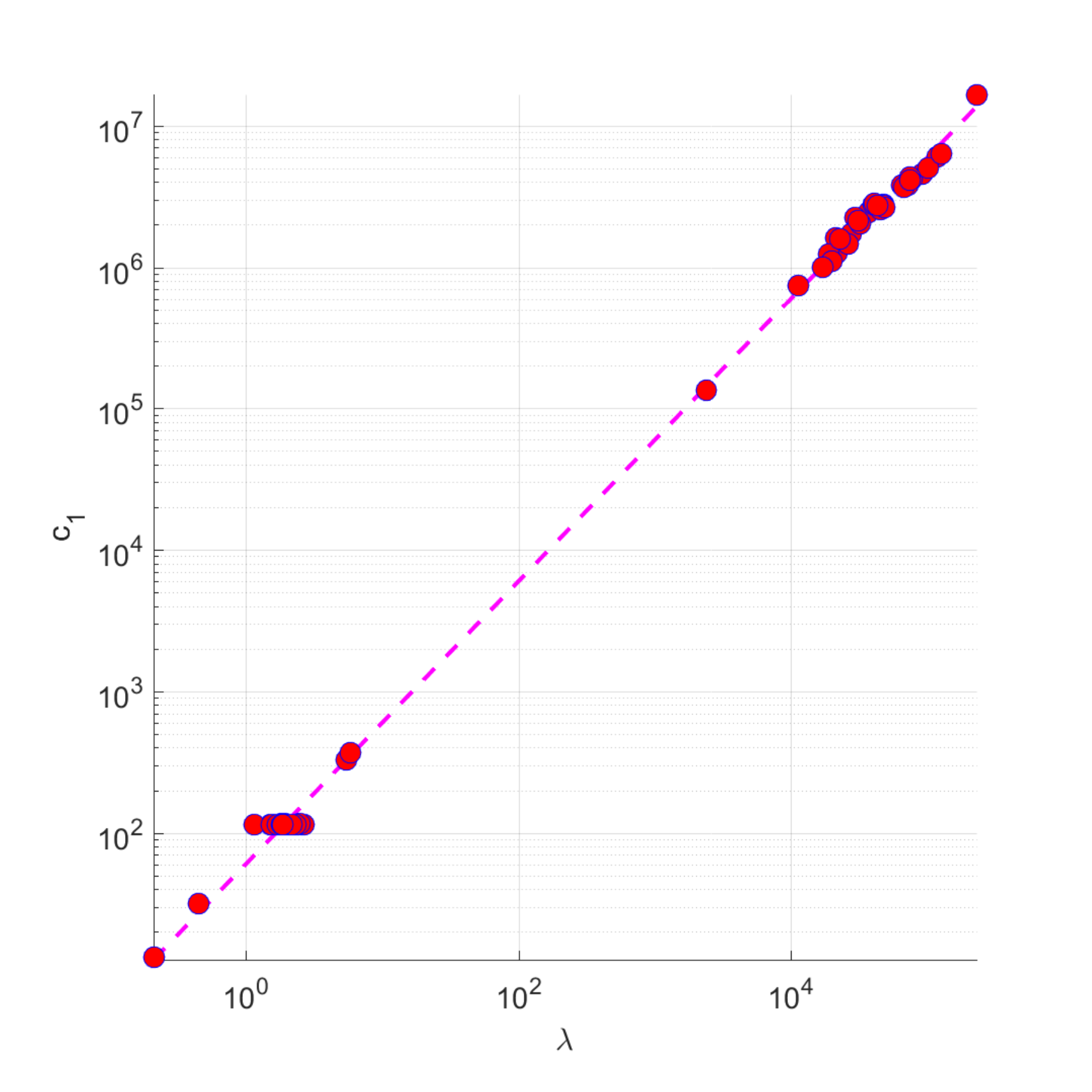}
\caption{$c_{1}$ vs. $\lambda$}
\label{c1-la}
 \end{subfigure}
\caption{SLI $c_{1}$ and $\lambda$ parameters estimated from 100 small training sets that include 10\% of the data. (a) Values of the $c_1$ parameter. (b) Scatter plot of $c_1$ versus $\lambda$ (logarithmic axes are used).}
 \label{fig:train_10_SLI}
 \end{figure}

In order to develop some understanding for this behavior, let us recast the SLI predictive equations~\eqref{eq:sli-prediction} and~\eqref{eq:sli-J-p} as follows

\beq
\label{eq:sli-prediction-free}
\hat{\xo}_{p}  = \mx - \frac{1}{ N^{-1} + c_{1}^\ast \, W_{p}} \, \sum_{n=1}^{N}  \,  c_{1}^\ast \, \left( w_{n,p} + w_{p,n} \right) \, (\xo_{n} - \mx)
\eeq
where
\[
W_{p} =\sum_{n=1}^{N} \left( w_{n,p} + w_{p,n} \right).
\]
Equation~\eqref{eq:sli-prediction-free} shows that the prediction is independent of $\lambda$. In addition, if $N^{-1} \ll c_{1}^\ast \, W_{p}$, the prediction essentially remains invariant if
$c_{1}^\ast$ is replaced by $\alpha c_{1}^\ast$, since the factors $c_{1}^\ast$ in the numerator and the denominator cancel out.
Then, SLI energy is also multiplied by $\alpha$. Finally, the new value of $\lambda$ according to~\eqref{eq:lambda} is also multiplied with $\alpha$ since it is proportional to the energy. These arguments show that different values of $c_{1}^\ast$ and $\lambda^\ast$ with constant  ratio $c_{1}^\ast/\lambda^\ast$ lead to nearly identical predictions, provided that the condition $N\, c_{1}^\ast \, W_{p} \gg 1$ is fulfilled.
Exact proportionality only holds if the above condition is fulfilled for all the prediction points.

We repeat the same experiment using large training sets that include $90\%$ (i.e., $10278$) of the points. In this case the estimates of  $c_1$ are close to the initial guess $c_1=115$ for all the subsets. Similarly, $\lambda$ estimates are in the range   $[0.184,0.200]$.
The results are shown in Fig.~\ref{fig:train_90_SLI}. In this case the estimates of $c_{1}, \lambda$ are considerably more stable and there is no apparent trend. This  is explained by the higher sampling density of larger training sets, and the fact that the training sets share many points (since $90\%$ of the total number of points are used in each training set).

\begin{figure}[h!]
\centering
\begin{subfigure}[b]{0.49\textwidth}
\includegraphics[width=\linewidth]{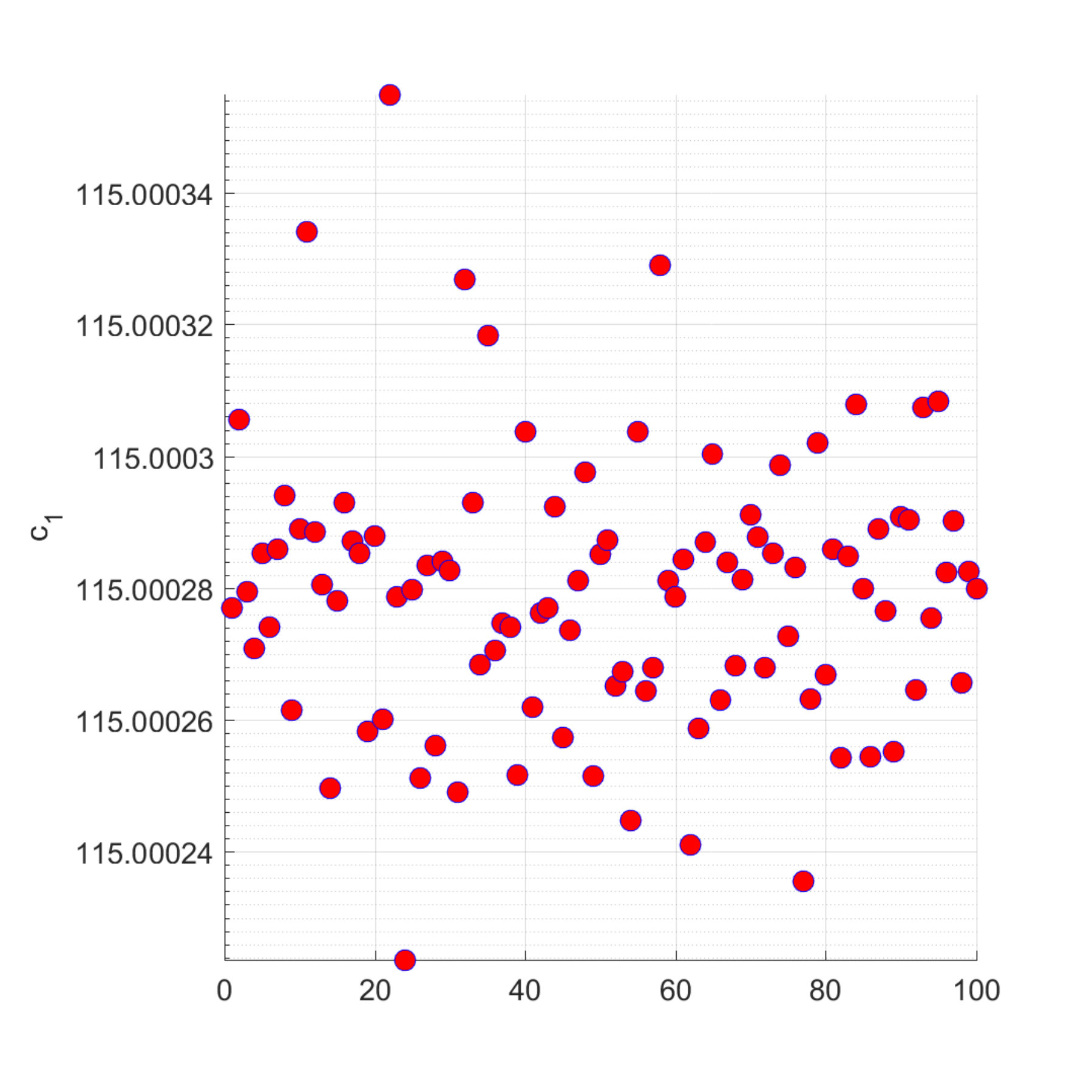}
\caption{$c_1$}
\label{c1-90}
 \end{subfigure}
\begin{subfigure}[b]{0.49\textwidth}
\includegraphics[width=\linewidth]{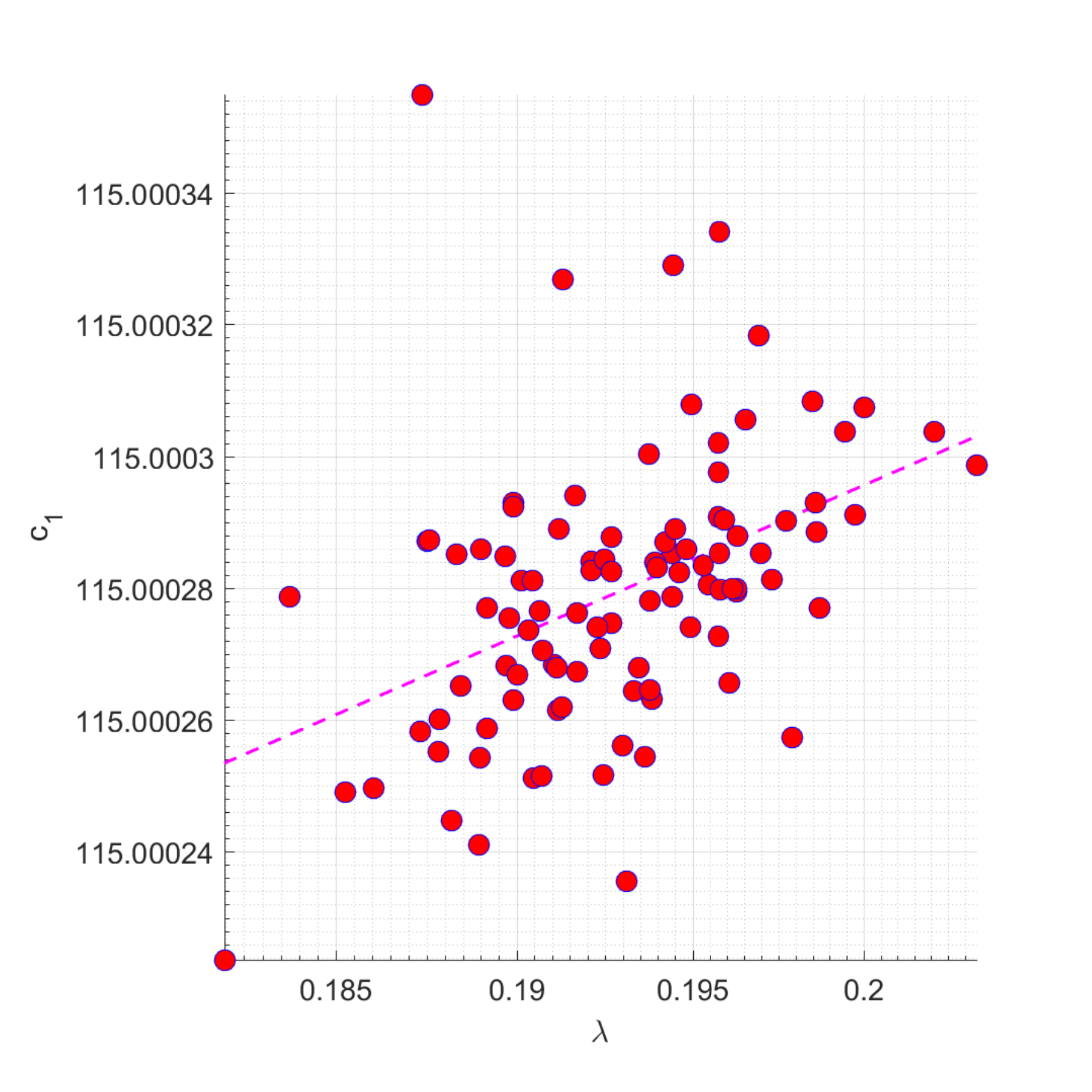}
\caption{$c_1$ vs. $\lambda$}
\label{c1-la-90}
 \end{subfigure}
\caption{SLI $c_{1}$ and $\lambda$ parameters estimated from 100 large training sets that include 90\% of the data. (a) Values of the $c_1$ parameter. (b) Scatter plot of $c_1$ versus $\lambda$.}
 \label{fig:train_90_SLI}
 \end{figure}

 %\tr{STOPPED HERE}

 %===================================================
\subsection{Comparison of Cross-Validation Statistics}

In this section we compare the interpolation performance of SLI and OK in terms of the cross-validation statistics, using results from both the small and large training sets.
In the SLI analysis the spherical kernel with neighbor order  $k=3$ is used while in OK the spherical model with nugget is fitted to the empirical variogram and a kriging search radius of 9.1~km is used to reduce computational time.  This radius is sufficient for cross-validation analysis.

In Fig.~\ref{fig:boxplots-train-10-cv} we compare the box plots for ME, MAE, RMSE and R for the two methods based on cross validation analysis of small training sets. In addition,  Fig.~\ref{fig:boxplots-train-90-cv} extends the comparison to large training sets. The main patterns evidenced in the box plots between the SLI and the OK cross validation statistics are the following:

\begin{enumerate}

\item The SLI MAE and RMSE have lower median values than the respective OK measures (for both small and large training sets).
\item The correlation coefficient has a higher median value for  SLI than for OK (also for both small and large training sets).
\item In the case of small training sets, the ME median for OK is closer to zero (0.0055) than for SLI (0.0155).  In addition, the inter-quartile range is higher for SLI (0.0898) than for OK (0.0770).
\item This pattern is reversed for large training sets: the OK-based ME median is $-0.0050$ while the SLI-based ME median is $-0.0011$. The inter-quartile range is higher for OK ($0.0781$)  than for SLI ($0.0702$). So, it seems that SLI makes more efficient use of the higher sampling density.

\end{enumerate}

\begin{figure}
\centering
\includegraphics[width=0.9\textwidth]{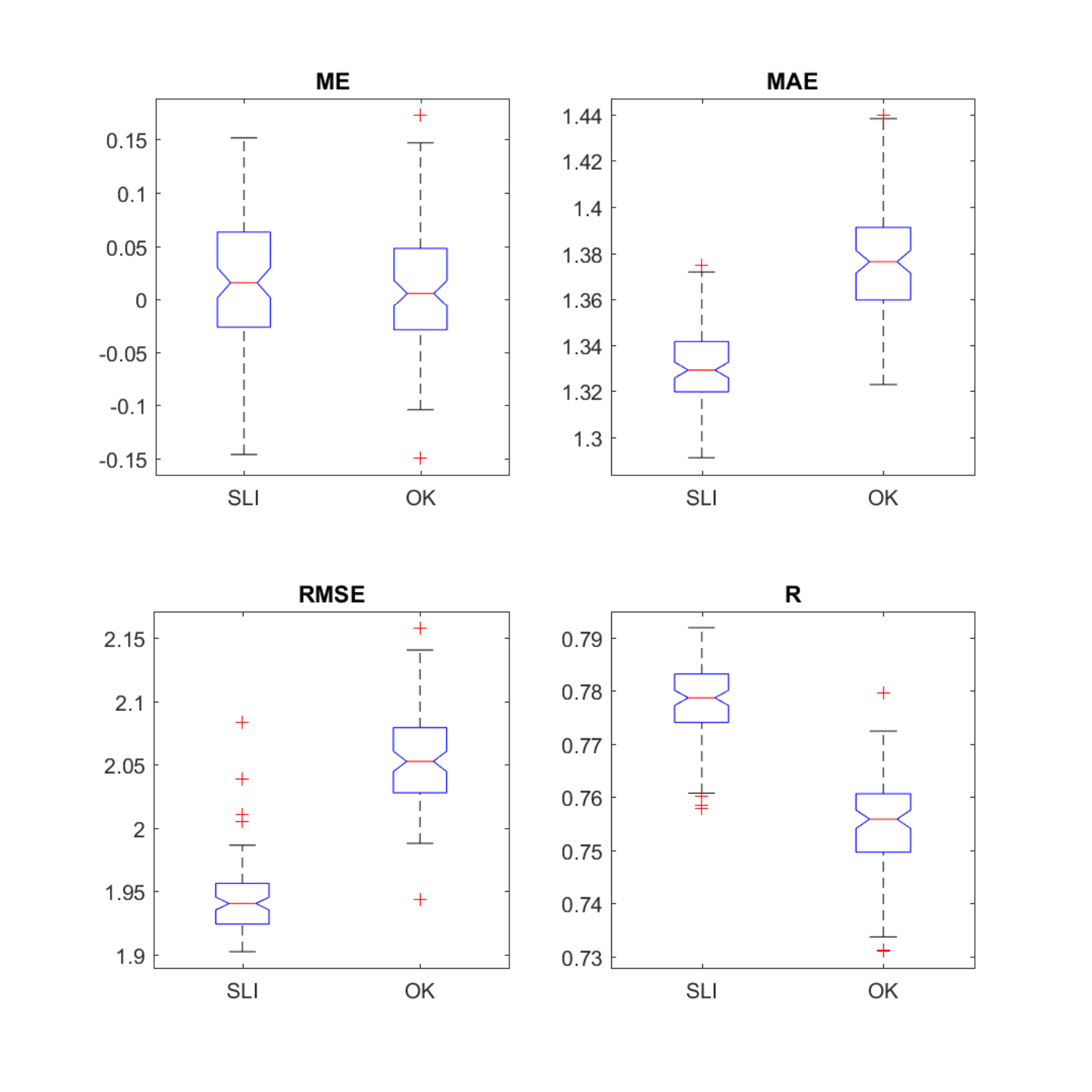}
\caption{Box plots of cross-validation statistics for the SLI and OK methods. The CV statistics are calculated using a training set that contains 10\% of the data points and a validation set that comprises 90\% of the points. }
\label{fig:boxplots-train-10-cv}
\end{figure}

\begin{figure}
\centering
\includegraphics[width=0.9\textwidth]{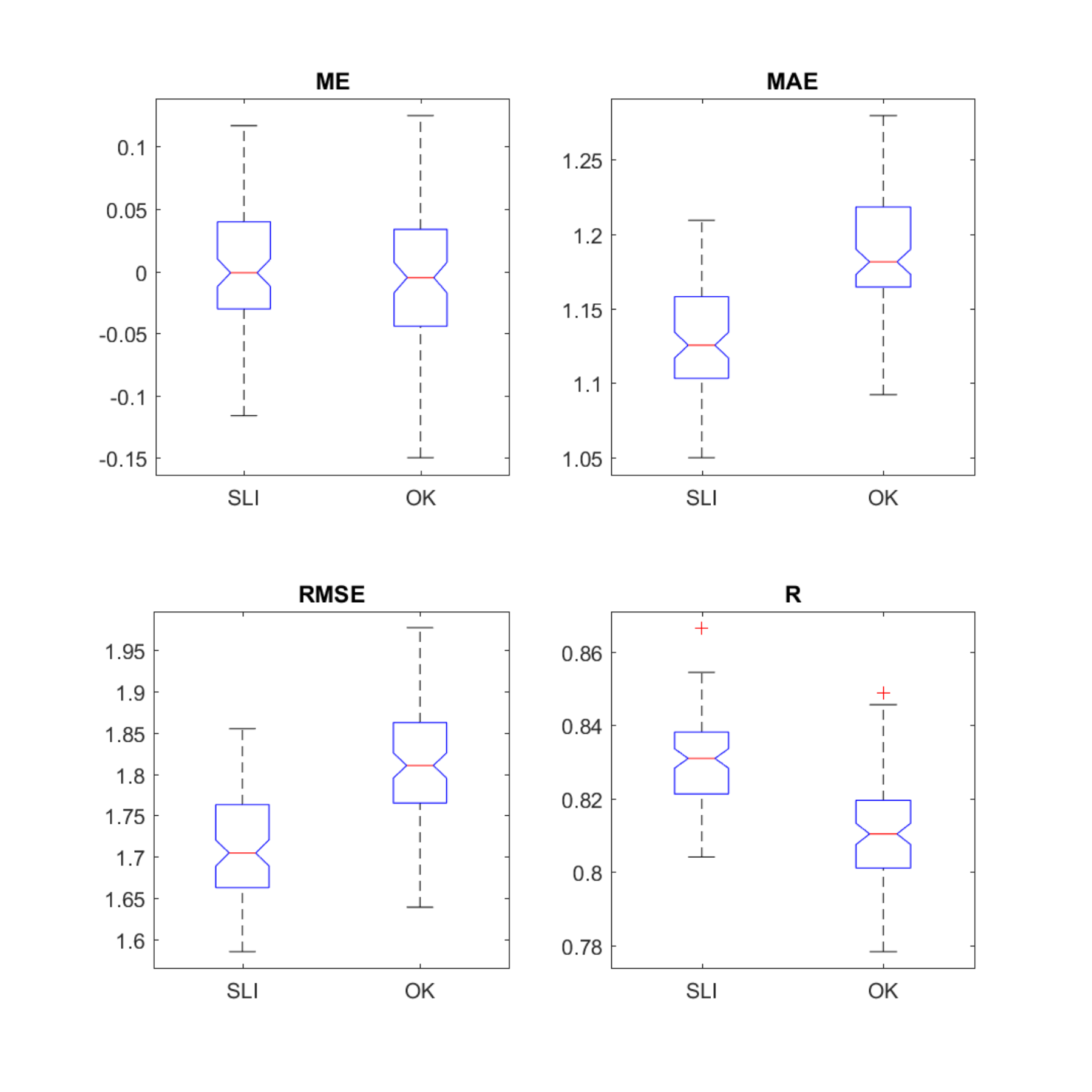}
\caption{Box plots of cross-validation statistics for the SLI and OK methods. The CV statistics are calculated using a training set that contains 90\% of the data points and a validation set that comprises 10\% of the points. }
\label{fig:boxplots-train-90-cv}
\end{figure}

\clearpage

\section{Conclusions}
\label{sec:Conclusions}

This paper presents a stochastic local interaction (SLI) model for spatial estimation which is based on a generalized concept of the square gradient for scattered data. The SLI spatial prediction is expressed, conditionally on the data, in terms of a \emph{sparse precision} matrix, which is demonstrably positive definite. The  precision matrix allows \emph{fast estimation} of the mean and variance of the conditional  distribution at the prediction locations. The sparse SLI precision matrix is a feature analogous to that of Gaussian Markov random fields defined on regular grids. Hence, SLI extends the original GMRF idea  by means of kernel functions that implement local interactions between the nodes of  unstructured networks. This implementation is suitable for modeling scattered spatial data. In addition, SLI allows efficient processing of large datasets, because it incorporates spatial correlations in the form  of sparse precision matrices. For example, for the current study the sparsity index is $\approx 0.21\%$, i.e., less than 1\% of the precision matrix elements are non-zero. The sparsity feature implies considerably reduced memory space requirements.  In addition, matrix inversion is not required to derive the SLI estimates and their standard deviations at unmeasured locations, thus reducing the necessary computational time.

With respect to the analysis of the Campbell coal dataset, SLI gives comparable (slightly better) cross-validation statistics compared to those obtained with  Ordinary Kriging with a finite search radius. In addition, SLI estimates a higher volume of coal in the study area than OK. Regarding the prediction uncertainty, the SLI  estimates of conditional variance are overall higher  than those of  Ordinary Kriging. This is not surprising, since the SLI predictor is not based on optimization of the error variance. In terms of computational speed, filling the interpolation grid with the SLI predictor is about 3-25 times faster than with ordinary kriging ($\approx$ 12 min with SLI versus $\approx$ 42 min with a search radius of 9.1 km and $\approx$ 5 hr with a search radius of 14.6 km). 
Overall, we have shown that SLI is a competitive method for the spatial interpolation of scattered data which can be applied to the estimation of natural resources, providing computational benefits for large datasets. In conclusion, we believe that SLI will provide a competitive tool for spatial analysis of large spatial datasets.

In terms of future research,  anisotropic and non-stationary  extensions of SLI are possible, as well as the development of SLI-based conditional simulation capabilities that can take advantage of the sparse precision matrix structure.  Geometric anisotropy can be handled using a directional measure of distance  in the  kernel functions. Non-stationarity can be addressed in two ways: firstly, by introducing a variable mean function $\mx(\bfs)$ instead of a constant $\mx$ in the SLI energy~\eqref{eq:sli-ene}, and secondly by replacing the rigidity parameter $c_{1}$ with a rigidity function $c_{1}(\bfs)>0$. The application of SLI to  spatiotemporal data is also possible,  and a first step in this direction is presented in~\cite{dth19}. More general SLI models with additional terms in the energy function and multiple variates can also be developed. Finally, the positive-definiteness of the precision matrix is maintained even if non-Euclidean measures of distance are used in the kernel functions. This is a potential advantage compared to standard geostatistical methods which are based on  covariance function models, since the latter are not necessarily permissible for non-Euclidean distance metrics.

\newpage

\section*{Acknowledgment}
We would like to thank Ricardo Olea (United States Geological Survey, Reston, Virginia, USA) who provided  the coal data analyzed and graciously read a draft of this manuscript, offering valuable comments and insights.

%\bibliographystyle{MG}       % Mathematical Geoscience style

\end{document}